\documentclass[10pt]{article}

\usepackage{microtype}
\usepackage{graphicx}
\usepackage{subcaption}
\usepackage{booktabs}

\pdfoutput=1

\usepackage{amsmath,amsbsy,amsgen,amscd,amssymb,amsthm,amsfonts,stmaryrd}
\usepackage{mathtools}

\usepackage{bm}
\usepackage{enumerate}

\usepackage{microtype}      

\usepackage[colorlinks,citecolor=blue]{hyperref}
\usepackage{cleveref}
\usepackage{url}            

\usepackage[usenames,dvipsnames,svgnames]{xcolor}
\usepackage{graphicx}

\graphicspath{{art/}}

\definecolor{dark-gray}{gray}{0.3}
\definecolor{dkgray}{rgb}{.4,.4,.4}
\definecolor{dkblue}{rgb}{0,0,.5}
\definecolor{medblue}{rgb}{0,0,.75}
\definecolor{rust}{rgb}{0.7,0.1,0.1}
\definecolor{drust}{rgb}{0.5,0.1,0.1}

\hypersetup{urlcolor=drust}
\hypersetup{citecolor=blue}
\hypersetup{linkcolor=rust}

\newtheorem{theorem}{Theorem}
\newtheorem{lemma}[theorem]{Lemma}
\newtheorem{corollary}[theorem]{Corollary}
\newtheorem{fact}[theorem]{Fact}

\numberwithin{equation}{section}
\numberwithin{theorem}{section}

\renewcommand{\phi}{\varphi}

\newcommand{\proj}{\operatorname{proj}}

\newcommand{\cG}{\mathcal{G}}
\newcommand{\bx}{\mathbf{x}}
\newcommand{\by}{\mathbf{y}}
\newcommand{\bz}{\mathbf{z}}
\newcommand{\bu}{\mathbf{u}}
\newcommand{\bv}{\mathbf{v}}

\newcommand{\bb}{\mathbf{b}}
\newcommand{\ba}{\mathbf{a}}
\newcommand{\zinit}{\mathring{\mathbf{z}}}
\newcommand{\Gap}{\mathrm{Gap}}
\newcommand{\Ghat}{\widehat{G}}
\newcommand{\Lhat}{\widehat{L}}

\DeclareFontFamily{OT1}{pzc}{}
\DeclareFontShape{OT1}{pzc}{m}{it}{<-> s * [1.200] pzcmi7t}{}
\DeclareMathAlphabet{\mathpzc}{OT1}{pzc}{m}{it}

\DeclareMathOperator{\prox}{prox}

\usepackage{xcolor,colortbl}

\usepackage[most]{tcolorbox}

\newtheorem{assumption}{Assumption}

\tcolorboxenvironment{assumption}{
  enhanced,
  breakable,
  colback=gray!10,
  colframe=gray!50,
  boxrule=0.5pt,
  arc=2mm,
  left=8pt,
  right=8pt,
  top=4pt,
  bottom=4pt,
}

\makeatletter
\newtheorem*{rep@theorem}{\rep@title}
\newcommand{\newreptheorem}[2]{%
	\newenvironment{rep#1}[1]{%
		\def\rep@title{#2 \ref{##1}}%
		\begin{rep@theorem}}%
		{\end{rep@theorem}}}
\makeatother

\newreptheorem{theorem}{Theorem}
\newreptheorem{lemma}{Lemma}

\theoremstyle{definition}
\newtheorem{remark}[theorem]{Remark}

\usepackage{lscape}
\usepackage{makecell}

\usepackage{natbib}
\usepackage{algorithm}
\usepackage{algorithmic}
\usepackage[margin=1in]{geometry}

\crefname{theorem}{Theorem}{Theorems}
\Crefname{theorem}{Theorem}{Theorems}
\crefname{proposition}{Proposition}{Propositions}
\Crefname{proposition}{Proposition}{Propositions}
\crefname{lemma}{Lemma}{Lemmas}
\Crefname{lemma}{Lemma}{Lemmas}
\crefname{corollary}{Corollary}{Corollaries}
\Crefname{corollary}{Corollary}{Corollaries}
\crefname{definition}{Definition}{Definitions}
\Crefname{definition}{Definition}{Definitions}
\crefname{assumption}{Assumption}{Assumptions}
\Crefname{assumption}{Assumption}{Assumptions}
\crefname{remark}{Remark}{Remarks}
\Crefname{remark}{Remark}{Remarks}

\crefname{appendix}{Appendix}{Appendices}
\Crefname{appendix}{Appendix}{Appendices}
\crefname{appsubsection}{Appendix}{Appendices}
\Crefname{appsubsection}{Appendix}{Appendices}

\newcommand{\E}{\mathbb{E}}

\usepackage[colorinlistoftodos,prependcaption]{todonotes}

\title{How to Make the Gradient Mapping Small for Constrained Stochastic Min-Max Problems and Beyond}
\author{Ahmet Alacaoglu\footnote{Department of Mathematics, University of British Columbia, \url{ahmet.alacaoglu@ubc.ca} }}
\date{}
\begin{document}
\maketitle
\begin{abstract}
We study the stochastic first-order oracle complexity for constrained or regularized convex-concave min-max optimization and stochastic monotone variational inequalities. We focus on the case when suboptimality is measured in terms of the gradient mapping, also known as, forward-backward or natural residual, an optimality notion that generalizes the gradient norm for unconstrained problems. 
In this setting,  under standard unbiased oracle access with now-standard variance assumptions, the best-known complexity for making the norm of the gradient mapping less than $\varepsilon$ is $\widetilde{O}(\varepsilon^{-4})$, compared to the near-optimal $\widetilde{O}(\varepsilon^{-2})$ that is established in the unconstrained case. We bridge this gap to improve the gradient mapping complexity for constrained convex-concave min-max problems to $\widetilde{O}(\varepsilon^{-2})$. We then extend to prove the same complexity for problems without the bounded variance, by using the Blum-Gladyshev assumption.
\end{abstract}
\section{Introduction}
We study algorithms to solve the standard template of constrained convex-concave min-max optimization, 
\begin{equation}\label{eq: min-max}
\min_{\bu \in U} \max_{\bv \in V} f(\bu, \bv),
\end{equation}
where $f(\cdot, \bv)$ is convex, $f(\bu, \cdot)$ is concave, the gradients $\nabla_{\bu} f(\bu, \bv), \nabla_{\bv} f(\bu, \bv)$ are Lipschitz continuous, and $U\subseteq \mathbb{R}^n, V\subseteq \mathbb{R}^m$ are closed convex sets admitting efficient projection operators.

Two canonical examples of this template come from cases where the coupling between $\bu$ and $\bv$ is bilinear, such as matrix games
\begin{equation}\label{eq: mat_games}
\min_{\bu\in\Delta_n} \max_{\bv\in\Delta_m} \langle A\bu, \bv\rangle,
\end{equation}
where $\Delta_n$ denotes the probability simplex, that is, $\Delta_n=\{ \bu\in\mathbb{R}^n_+\colon \sum_{i=1}^n u_i = 1\}$.

The second common example is linearly constrained optimization
\begin{equation*}
\min_{\bu\in \mathbb{R}^n} f(\bu): A\bu\leq \bb,
\end{equation*}
for a proper, convex and closed function $f\colon\mathbb{R}^{n}\to(-\infty, \infty]$. Solving this problem efficiently is generally done by utilizing convex duality to convert to a min-max problem given as
\begin{equation}\label{eq: lin_const}
\min_{\bu \in \mathbb{R}^n} \max_{\bv\geq 0} f(\bu) + \langle A\bu-\bb, \bv \rangle.
\end{equation}
Strictly speaking, this problem does not fit the template \eqref{eq: min-max} due to potential nonsmoothness in $f$. However, our focus on \eqref{eq: min-max} in this section is purely for a simple presentation. In the sequel, we develop our results for the more general problem of \emph{variational inequalities} to cover this example, see \eqref{eq: vi} for a precise problem statement.

We will focus on a setting where we do not have access to gradients $\nabla_\bu f(\bu, \bv)$ and $\nabla_{\bv}f(\bu, \bv)$, but to their unbiased estimators such that
\begin{equation*}
\mathbb{E}[\nabla_{\bu}f_\xi(\bu, \bv)] =\nabla_{\bu} f(\bu, \bv),
\end{equation*}
where $\xi$ is sampled from an unknown distribution $P$, where we can have $f(\bu, \bv) = \mathbb{E}_{\xi\sim P}[f_{\xi}(\bu, \bv)]$. At every iteration of our algorithm, we will assume that we receive i.i.d. samples $\xi \sim P$.

More concretely, for \eqref{eq: mat_games} or \eqref{eq: lin_const}, these stochastic oracles may correspond to randomly sampling rows or columns from the matrix $A$, whereas the full-gradients require accessing the matrix $A$ at every iteration. A concrete example is given at the end of Section \ref{sec: prelim}.
\paragraph{Two optimality notions. }
Arguably, the most widely used optimality notion for min-max problems is the primal-dual gap
\begin{equation*}
\Gap(\bar \bu, \bar \bv) = \max_{\bu \in U, \bv \in V} f(\bar \bu, \bv) - f(\bu, \bar\bv).
\end{equation*}
This is well-suited to problems such as matrix games in \eqref{eq: mat_games}, however, when the sets $U, V$ are unbounded, such as the case of linearly constrained optimization in \eqref{eq: lin_const}, this notion runs into problems. In particular, this quantity may be infinite at test points $\bar\bu, \bar\bv$ that are not optimal. 

This well-known drawback gave rise to the restricted primal-dual gap function \citep{nesterov2009primal} that takes the maximum over a compact set that needs to satisfy certain requirements. That is, this set needs to contain a solution and the iterates of the algorithm, see \citep[Lemma 4]{nesterov2009primal}. This approach can be compatible in the deterministic case when the boundedness of the iterates can be shown for standard methods. 

For stochastic algorithms, boundedness of the iterates does not generally hold, causing difficulty in identifying a compact set to define the restricted gap function \'a la \cite{nesterov2009primal}. Moreover, the restricted gap is generally taken over sets that generally depend on the solution, making it difficult to compute in general.

Despite its incompatibility with unbounded domains and stochastic problems, this notion is widely used because the standard convergence analyses for algorithms, such as extragradient \citep{korpelevich1976extragradient} or forward-backward-forward (FBF) methods \citep{tseng2000modified}, give a rate result on the gap as a byproduct. As a result, optimal methods for reducing the gap is well-established for both deterministic and stochastic problems \citep{nemirovski2004prox,nemirovski2009robust}.

Perhaps even a more natural way of measuring suboptimality for unbounded problems is the gradient mapping norm \citep[Section 2]{nesterov2013gradient}, which generalizes the gradient norm for unconstrained problems
\begin{equation}\label{eq: grad_mapping_f}
\| \cG_{\eta, f}(\bx, \by)\|^2 = \frac{1}{\eta^2}\left(\| \bx-\proj_U(\bx-\eta \nabla_\bx f(\bx, \by))\|^2 + \|\by-\proj_V(\by+\eta \nabla_\by f(\bx, \by))\|^2\right),
\end{equation}
where $\proj_C$ denotes the Euclidean projection on $C$.
While this notion is widely used for nonconvex problems, optimal complexity results for making this quantity small, even for convex minimization problems is relatively recent, see \cite{nesterov2012make} for deterministic and \cite{allen2018make} for stochastic optimization.

Clearly, this reduces to the gradient norm when $(U, V)=(\mathbb{R}^n, \mathbb{R}^m)$ since 
\begin{equation*}
\cG_{\eta, f} (\bx, \by) = \binom{\nabla_\bx f(\bx, \by)}{-\nabla_{\by}f(\bx, \by)}.
\end{equation*}
Optimal methods for reducing the gradient -- or gradient mapping -- norm for deterministic convex-concave min-max problems is established relatively recently as well, see \cite{diakonikolas2020halpern,kim2021accelerated,yoon2021accelerated,kovalev2022first}.
An optimal method for unconstrained and stochastic min-max problems is proposed in \cite{chen2024near} who showed the complexity $\widetilde{O}(\varepsilon^{-2})$ for making the gradient norm less than $\varepsilon$ with matching lower bounds. 

Yet, the best-known complexity for constrained min-max problems, which are of interest in terms of many applications -- such as matrix games \eqref{eq: mat_games}, linearly constrained optimization \eqref{eq: lin_const}, and distributionally robust optimization \citep[Eq. (1)]{namkoong2016stochastic} -- remained at $\widetilde{O}(\varepsilon^{-4})$ -- with the exception of \cite{trandinh2026unbiasedbiasedvariancereducedforwardreflectedbackward} that improved this to $\widetilde{O}(\varepsilon^{-10/3})$ under stronger assumptions on the stochastic oracle, see Table \ref{tab:constraints1}. The goal of this work is to improve this complexity to one that is optimal in terms of $\varepsilon$ dependence. This discussion is summarized in Table \ref{tab:constraints2} where we highlighted the main setting we focus on. 

\paragraph{Generalized problems. }
For a simpler notation and a slight generalization, we will consider the problem of variational inequalities where the goal is to
\begin{equation}\tag{VI}
\text{find~} \bx^\star \text{~such that~} \langle F(\bx^\star), \bx-\bx^\star \rangle \geq r(\bx^\star) - r(\bx)\text{~for any~} \bx\in\mathbb{R}^d,\label{eq: vi}
\end{equation}
with a proper, convex and closed $r\colon\mathbb{R}^d\to(-\infty, \infty]$ and a \emph{monotone operator}\footnote{Introduced in Section \ref{sec: prelim}.} $F\colon\mathbb{R}^d\to\mathbb{R}^d$. Mapping \eqref{eq: min-max} to this problem is done by setting
\begin{equation}\label{eq: mapping_minmax_vi}
\bx=\binom{\bu}{\bv}, ~~~ F(\bx) = \binom{\nabla_{\bu} f(\bu, \bv)}{-\nabla_{\bv} f(\bu, \bv)}, ~~~ r(\bx) = \delta_{U}(\bu) + \delta_V(\bv),
\end{equation}
where $\delta_C(\bx)$ denotes the indicator function, that is $0$ when $\bx\in C$ and $\infty$ otherwise.

We will often use the equivalent form written as a \emph{monotone inclusion} problem where the goal is to
\begin{equation}\tag{MI}\label{eq: mi}
\text{find~} \bx^\star \text{~such that~} 0 \in H(\bx^\star):=(F+\partial r)(\bx^\star).
\end{equation}
We focus on this specific monotone inclusion for presentation purposes, but by changing \Cref{alg:eg_anchor} to a method based on FBF \citep{tseng2000modified} or forward-reflected-backward method \citep{malitsky2020forward}, our analysis can be extended to solve general monotone inclusion where $\partial r$ is replaced by a maximally monotone operator. We skip this extension for keeping our notation simpler.

By overloading the terminology for simplicity, we shall now write the \emph{gradient mapping} as
\begin{equation*}
\cG_{\eta, H}(\bx) = \frac{1}{\eta} \left( \bx-\prox_{\eta r}(\bx-\eta F(\bx)) \right),
\end{equation*}
which is equivalent to \eqref{eq: grad_mapping_f} with the setting \eqref{eq: mapping_minmax_vi}. We emphasize that $F$ does not need to be the gradient of a function, and we merely use this notation to keep the link with min-max optimization \eqref{eq: min-max} transparent. This notion is also referred to as forward-backward residual or natural residual (in the specific case when $r$ is an indicator function), see \citep[Section 10.3]{facchinei2003finite}. 

Our goal in the sequel is to find an approximate solution to \eqref{eq: mi} by making the residual small. This result directly covers making the gradient mapping small for convex-concave and constrained min-max problems by the mapping \eqref{eq: mapping_minmax_vi}.

\paragraph{Complexity. } Given the problem \eqref{eq: vi} introduced above, we measure complexity in terms of the number of calls to the \emph{unbiased stochastic oracles} $F_\xi$ such that
\begin{equation*}
\mathbb{E}[F_\xi(\bx)] = F(\bx)
\end{equation*}
and to the proximal operator of $r$ ---which reduces to the number of calls to stochastic gradients and projections on $U, V$ in the case of min-max problems \eqref{eq: min-max}--- to get a certain optimality notion less than $\varepsilon$.

There exists now a large body of works that analyze methods such as stochastic extragradient with optimal complexity for reducing the gap. That is, it is well-known how to get
\begin{equation*}
\mathbb{E}[\Gap(\bar\bx, \bar\by)] \leq \varepsilon, \text{~with complexity~} O(\varepsilon^{-2}).
\end{equation*}
We refer to \cite{nemirovski2009robust,juditsky2011solving,kotsalis2022simple,zhao2022accelerated,lan2026novel} for representative examples.

Moreover, when $r\equiv 0$, that is, we have a root finding problem, \cite{chen2024near} established the oracle complexity
\begin{equation*}
\widetilde{O}(\varepsilon^{-2}) \text{~for making~}\mathbb{E}\|F(\bx)\| \leq \varepsilon,
\end{equation*}
which is optimal. The authors assumed that the unbiased oracle has bounded variance upper bounded by $\sigma^2$ and designed an algorithm that sets its parameters by using $\sigma^2$ and $\|\bx_0-\bx^\star\|^2$, the initial distance of their initial point to a solution. For unconstrained problems, with stronger assumptions such as sharpness, \cite{cai2022stochastic} also obtained a similar complexity, yet for unconstrained and merely monotone problems, the complexity in this work deteriorated to $\widetilde{O}(\varepsilon^{-3})$.

For constrained problems, the best-known complexity for the gradient mapping under our assumptions is
\begin{equation*}
O(\varepsilon^{-4}) \text{~for making~}\mathbb{E}\|\cG_{\eta, H}(\bx)\| \leq \varepsilon,
\end{equation*}
which is suboptimal. This complexity remained as the best-known on a large body of literature for constrained or regularized min-max problems \citep{iusem2017extragradient,boct2021minibatch,alacaoglu2025towards,lee2021fast,pethick2023solving,bohm2023solving,diakonikolas2021efficient,tran2025vfog}. Even though some of these works relied on some relaxations of convex-concavity, none of the analyses in those works gave rise to a better complexity for convex-concave problems. The aim of this work is to improve this complexity for problems satisfying the Blum-Gladyshev variance and without requiring parameters depending on initial distance to solution or variance upper bounds.

A recent work by \cite{trandinh2026unbiasedbiasedvariancereducedforwardreflectedbackward} obtained a complexity of $\widetilde{O}(\varepsilon^{-10/3})$ for constrained problems by using variance reduction. Different from our work, they required a multi-point oracle that requires querying the stochastic oracle at the same seed for different points, the mean-square Lipschitzness assumption (which is stronger than mere Lipschitzness of $F$) and bounded variance.

\begin{table*}[t!]
\centering
\begin{tabular}{l l l l l }
\toprule
&\textbf{Constraints} & \makecell{\textbf{Need for setting} \\ \textbf{ alg. parameters}} & \makecell{\textbf{Stoc. oracle} \\ \textbf{complexity}} & \textbf{Variance assumption} \\
\midrule
\cite{chen2024near} & $\times$ & $\sigma^2, \| \bx_0-\bx^\star\|^2, L_F, G$ & $\widetilde{O}(\varepsilon^{-2})$ & Bounded \\[3mm]
\cite{trandinh2026unbiasedbiasedvariancereducedforwardreflectedbackward} & $\checkmark$ & $L_F$ & $\widetilde{O}(\varepsilon^{-10/3})$ & \makecell[l]{Bounded variance \\ Multi-point oracle \\ mean-squared Lipschitz} \\[3mm]
\makecell[l]{\cite{iusem2017extragradient}\\and other works} & $\checkmark$ & $L_F, B$  & $\widetilde{O}(\varepsilon^{-4})$ & Blum-Gladyshev \\[3mm]
This work & $\checkmark$ & $L_F, B$ & $\widetilde{O}(\varepsilon^{-2})$ & Blum-Gladyshev\\
\bottomrule
\end{tabular}
\caption{Complexity results for stochastic convex-concave optimization under Assumption \ref{asp: 1} and variance assumption specified in the last column. For the first row, the complexity is for $\mathbb{E}\| F(\bx)\|\leq\varepsilon$. For the last three rows, the complexity is for $\mathbb{E}\| \cG_{\eta, H}(\bx)\|\leq\varepsilon$. See the notation in \eqref{eq: mapping_minmax_vi} and \eqref{eq: mi}.}
\label{tab:constraints1}
\end{table*}

\begin{table*}[t!]
\centering
\begin{tabular}{l|ll}
\toprule
\multicolumn{1}{c}{} 
& \multicolumn{1}{c}{\textbf{Constrained}} 
& \multicolumn{1}{c}{\textbf{Unconstrained}} \\
\midrule
\textbf{Gap} 
& \makecell{$O(\varepsilon^{-2})$\\\cite{juditsky2011solving}}
& \makecell{$O(\varepsilon^{-2})$ \\\cite{juditsky2011solving}} \\[5mm]

\textbf{Gradient mapping} 
& \fbox{\makecell{\color{gray} $O(\varepsilon^{-4})$\\\color{gray} \cite{kotsalis2022simple}}} 
& \makecell{$\widetilde{O}(\varepsilon^{-2})$ \\\cite{chen2024near}} \\[3mm]
\bottomrule
\end{tabular}
\caption{\small{Best-known complexity results for stochastic convex-concave optimization under Assumption \ref{asp: 1} and bounded variance. For the first row, the complexity is for making the (restricted) gap less than $\varepsilon$. For the second row, it is for making the gradient mapping (which reduces to gradient norm for unconstrained problem) less than $\varepsilon$.}}
\label{tab:constraints2}
\end{table*}

\subsection{Notation and Preliminaries}\label{sec: prelim}
Given the history of random vectors $\xi_0, \xi_{0.5}, \dots, \xi_k$, we denote the expectation conditioned on the generated $\sigma$-algebra as $\mathbb{E}_k[\cdot] = \mathbb{E}[\cdot~|~\sigma(\xi_0, \xi_{0.5}, \dots, \xi_k)]$. For a vector $\bx$, the notation $x_i$ refers to its $i$-th coordinate.

In this paper, we always consider a regularizer $r\colon\mathbb{R}^d \to (-\infty, \infty]$ that is convex, proper and closed. Its subdifferential set is denoted as $\partial r$. We define the proximal operator of $r$ as 
\begin{equation}\label{eq:prox}
\prox_{r}(\bx) =\arg\min_\by r(\by) + \frac{1}{2} \| \bx-\by\|^2,
\end{equation}
from which we can derive the \emph{prox-inequality}, by using the optimality condition of \eqref{eq:prox},
\begin{equation}\label{eq: prox_ineq}
\by=\prox_{r}(\bx) \iff \langle \by-\bx, \bu-\by \rangle \geq r(\by) - r(\bu)~~~\forall \bu\in\mathbb{R}^d.
\end{equation}
When $r(\bx) = \delta_{C}(\bx) = \begin{cases} 0, \text{~if~} \bx\in C,\\ \infty, \text{~if~} \bx\not\in C, \end{cases}$ this reduces to the projection operator
\begin{equation*}
\proj_{C}(\bx) =\arg\min_{\by\in C} \frac{1}{2} \| \bx-\by\|^2.
\end{equation*}
We say that an operator $F\colon\mathbb{R}^d \to\mathbb{R}^d$ is $\mu$-strongly monotone, when
\begin{equation*}
\langle F(\bx) - F(\by), \bx-\by \rangle \geq \mu\|\bx-\by\|^2,
\end{equation*}
and monotone when this inequality holds with $\mu=0$.

The operator $F$ is $L_F$-Lipschitz when
\begin{equation*}
\| F(\bx)-F(\by)\|\leq L_F\|\bx-\by\|.
\end{equation*}
Recalling our mapping $F=\binom{\nabla_\bu f(\bu, \bv)}{-\nabla_{\bv} f(\bu, \bv)}$ for a convex-concave $f$, this corresponds to Lipschitzness of the gradients of $\nabla_\bu f(\bu, \bv), \nabla_\bv f(\bu, \bv)$.

\begin{assumption}\label{asp: 1}
For problem \eqref{eq: mi}, let $F$ be monotone and $L_F$-Lipschitz. Assume that $r$ is proper, convex, closed. There exists $\bx^\star$ such that $0\in(F+\partial r)(\bx^\star)$. We have access to an unbiased estimator of $F$ such that $\mathbb{E}[F_\xi(\bx)] = F(\bx)$.
\end{assumption}
In view of the mapping \eqref{eq: mapping_minmax_vi}, \Cref{asp: 1} is satisfied when $f(\bu, \bv)$ is convex-concave (implies that $F$ is monotone), the gradients $\nabla_\bu f(\bu, \bv)$ and $\nabla_\bv f(\bu, \bv)$ are Lipschitz continuous (implies that $F$ is Lipschitz) and $U, V$ are convex closed sets (implies the requirements on $r$).

We now present the Blum-Gladyshev (BG) assumption from \cite{gladyshev1965stochastic,blum1954approximation} that is studied for minimization and  min-max problems recently.
\begin{assumption}[Blum-Gladyshev variance]\label{asp: 2}
Let the unbiased estimator of $F$ satisfy
\begin{equation*}
\mathbb{E} \| F_\xi(\bx) - F(\bx) \|^2 \leq B^2 \| \bx-\zinit\|^2 + G^2,
\end{equation*}
where $\zinit$ is a fixed initial point. In view of \Cref{asp: 1}, denote $\Lhat_F = L_F + B$.
\end{assumption}
The reason for selecting $\zinit$ as the center in the assumption is for convenience. One can change it to any other fixed center point, by absorbing the difference in the constant $G$ and slightly changing $B$.

 We also define the initial distance to a solution as (see \eqref{eq: mi})
\begin{equation}\label{eq:dstar}
D_{\star} = \| \zinit-\bx^\star\|.
\end{equation}

Let us first remark that this is a generalization of the commonly used bounded variance assumption that is used in \cite{allen2018make,chen2024near}. If $B=0$ in \Cref{asp: 2}, this would correspond to assuming a bounded variance.

In the sequel, we require the knowledge of $B$, but not $G$ for setting algorithmic parameters. As we explain now, $B$ is often connected to the Lipschitz constant and hence its knowledge should not be too restrictive in many cases.

Let us observe that a sufficient condition for the BG assumption is when we have $F(\bx) = \mathbb{E}[F_{\xi\sim P}(\bx)]$ and $F_\xi(\bx)$ is mean-square Lipschitz, that is, 
\begin{equation*}
\mathbb{E}\| F_\xi(\bx) - F_\xi(\by) \|^2 \leq L_{\exp}^2\|\bx-\by\|^2,
\end{equation*}
and $\mathbb{E} \| F_\xi(\bx^\star) - F(\bx^\star)\|^2\leq\sigma_\star^2$. This is sufficient because of the chain of inequalities
\begin{align*}
\mathbb{E} \| F_\xi(\bx) - F(\bx) \|^2 &\leq 2\mathbb{E} \| F_\xi(\bx) - F(\bx) - F_\xi(\bx^\star) + F(\bx^\star) \|^2 + 2\mathbb{E} \| F_\xi(\bx^\star) - F(\bx^\star)\|^2  \\
&\leq 2\mathbb{E} \|F_\xi(\bx) - F_\xi(\bx^\star) \|^2 + 2\mathbb{E}\| F_\xi(\bx^\star)-F(\bx^\star)\|^2\\
&\leq 2L_{\exp}^2\| \bx-\bx^\star\|^2+2\sigma_\star^2\\
&\leq \underbrace{4L_{\exp}^2}_{B^2} \| \bx-\zinit \|^2+\underbrace{4L_{\exp}^2 \| \zinit-\bx^\star\|^2 + 2\sigma_\star^2}_{G^2},
\end{align*}
where the first and last steps are by Young's inequality, the second step is using $\mathbb{E}[F_\xi(\bx) - F_\xi(\bx^\star)] = F(\bx) - F(\bx^\star)$ -- along with the inequality $\mathbb{E}\|X-\mathbb{E}X\|^2 \leq \mathbb{E}\|X\|^2$ for random vectors $X$. This proves that, for example, \citep[Assumptions 2, 4]{mishchenko2020revisiting} implies \Cref{asp: 2}.

A more concrete example is the bilinearly coupled min-max problem where a common oracle is by sampling row-column pairs from matrix $A\in\mathbb{R}^{m\times n}$:
\begin{equation*}
F_{\xi}(\bx) = \binom{mA_{j:} v_j }{-nA_{:i} u_i},
\end{equation*}
where $\xi=(i, j)$ are selected uniformly at random, $A_{:i}$ is $i$-th column of $A$ and $A_{j:}$ is $j$-th row of $A$.

Then, the variance $\mathbb{E}_{i, j} \left\| \binom{mA_{j:} v_j }{-nA_{:i} u_i}- \binom{A^\top v}{-A\bu} \right\|^2$ will not be bounded in general, unless $\bu, \bv$ live on compact sets -- an unrealistic assumption in general, preventing applying the results of the form \cite{chen2024near} even for unconstrained problems. Yet, the variance will scale quadratically in the norm of $\bx=\binom{\bu}{\bv}$, satisfying the BG assumption. This example can be generalized to solve problems with other affine operators and stochastic oracles.

Byproducts of our analysis also include optimal complexity guarantees for the gradient mapping norm for strongly convex-strongly concave problems. We also design algorithms not requiring difficult-to-compute constants such as the initial distance to solution, that is, $\|\zinit-\bx^\star\|^2$, required in \citep{chen2024near}, or a global variance upper bound, required in the work \citep{chen2024near}. 

\section{Algorithm \& Statement of Main Results}
\begin{algorithm*}[t]
\caption{Anchored stochastic extragradient -- $\mathtt{ASEG}(A+\partial r, \bx_0, \beta, \alpha, K)$}
\begin{algorithmic}
    \STATE Initial iterate $\bx_0$, global initial point $\zinit$, anchoring parameter $\beta$, step size $\alpha>0$, unbiased evaluations of $A$ denoted as $A_\xi$ \\
    \STATE Denote $\bx_{i} \equiv \bx_{i}^{s, n}$ for $i=0, \dots, K$
    \vspace{.2cm}
    \FOR{$k = 0, 1, 2,\ldots, K-1 $}
	\STATE $\bar\bx_k = \beta \zinit + (1-\beta) \bx_k$
        \STATE $\bx_{k+1/2} = \prox_{\alpha r}(\bar\bx_k - \alpha A_{\xi_k}(\bx_k))$
        \STATE $\bx_{k+1} = \prox_{\alpha r}(\bar\bx_k - \alpha A_{\xi_{k+1/2}}(\bx_{k+1/2}))$
        \ENDFOR
        \STATE \textbf{Output:} $\by_{n}^s := \hat\bx_K^{s, n} = \frac{1}{K}\sum_{k=0}^{K-1}\bx_{k+1/2}$.
      \end{algorithmic}
\label{alg:eg_anchor}
\end{algorithm*}

\begin{algorithm*}[t]
\caption{$\mathtt{SEG^{SC}}(A+\partial r, \by_0, \mu_A, \Lhat_A, B_A, T)$}
\begin{algorithmic}
    \STATE Initial iterate $\by_0$, oracle budget $T$, first period $N = \lfloor \frac{T}{8\Lhat_A/\mu_A} \rfloor$, second period $M = \lfloor \log_2\frac{T}{16\Lhat_A/\mu_A} \rfloor$ \\
    \STATE Denote $\by_{i} := \by_i^{s}$ for $i=0, 1,\dots, N+M$
    \vspace{.2cm}
    \FOR{$n = 1, 2,\ldots, N + M$}
    \IF{$n\leq N$}
        \STATE $\by_{n} = \mathtt{ASEG}(A+\partial r, \by_{n-1}, 3\alpha^2B_A^2, \alpha\equiv\frac{1}{2\Lhat_A}, \frac{2\Lhat_A}{\mu_A})$
\ELSE
        \STATE $\by_{n} = \mathtt{ASEG}(A+\partial r, \by_{n-1}, 3\alpha^2B_A^2, \alpha\equiv\frac{1}{2^{n-N}\Lhat_A}, \frac{2^{n-N+1}\Lhat_A}{\mu_A})$
        \ENDIF
     \ENDFOR
     \STATE \textbf{Output:} $\bz_{s} = \by_{N+M}^s$
      \end{algorithmic}
\label{alg:eg_sc}
\end{algorithm*}

\begin{algorithm*}[t]
\caption{Recursive regularization$(F^\mu+\partial r, \zinit, \mu, L_F, B, T)$}
\begin{algorithmic}
    \STATE Initial iterate $\bz_0=\zinit$, $\Lhat_F = L_F+B$, oracle budget $T\geq 48 \frac{\Lhat_F}{\mu} \left\lfloor\log_2\frac{\Lhat_F}{\mu}\right\rfloor$, $\Lhat_{s-1} = L_F+B+(2^s-1)\mu$ \\
    \STATE $F^0 = F^\mu$, $\mu_0=\mu$ \\
    \vspace{.2cm}
    \FOR{$s = 1, 2,\ldots, S=\lfloor \log_2\frac{\Lhat_F}{\mu} \rfloor $}
        \STATE $\bz_{s} = \mathtt{SEG^{SC}}(H^{s-1}:= F^{s-1}+\partial r, \bz_{s-1}, (2^s-1)\mu, \Lhat_{s-1}, B, T/S)$
        \STATE $\mu_s = 2\mu_{s-1}$
        \STATE $F^s(\bz) = F^{s-1}(\bz) + \mu_s(\bz-\bz_s)$
        \ENDFOR
      \end{algorithmic}
\label{alg:rec_reg}
\end{algorithm*}

The main algorithmic construction will be based on a key idea from \cite{allen2018make} who had focused on the same goal for composite convex optimization: \emph{recursive regularization}. Indeed, \cite{chen2024near} had also considered extension of \cite{allen2018make} to the min-max case. However, their extension came with three main drawbacks that we will address: (1) their analysis was limited to the unconstrained problem, (2) their analysis introduced a requirement for knowing the global variance upper bound and initial distance to optimum $\|\zinit-\bx^\star\|^2$, (3) they inherited the drawback from \cite{allen2018make} and required a globally bounded variance. We will go around these limitations by using the BG variance assumption and anchoring, proposed for a similar purpose in \cite{neu2024dealing} for proving a complexity result for the primal-dual gap.

\paragraph{Recursive regularization -- \Cref{alg:rec_reg}.} In our setting, this corresponds to solving strongly monotone inclusion problems with an increasing amount of added strong monotonicity. In particular,  this algorithm will give us an approximate solution for the problem
\begin{equation}\label{eq: delta_prob}
\text{find~} \bz^\star_0 \text{~such that~} 0\in H^\mu(\bz^\star_0) := F^\mu(\bz_0^\star) + \partial r(\bz^\star_0),
\end{equation}
where
\begin{equation}\label{eq: fdelta}
F^0(\bz) := F^\mu(\bz) = \begin{cases} F(\bz) \text{~if~} F\text{~is~}\mu\text{-strongly monotone,} \\
F(\bz) + \mu(\bz-\zinit) \text{~for some~} \mu>0 \text{~if~} F \text{~is monotone.} \end{cases}
\end{equation}
This setting ensures that $F^\mu$ is $\mu$-strongly monotone for some $\mu>0$, even if $F$ is only monotone.

Then, each step $s\geq 1$ of \Cref{alg:rec_reg} solves approximately the subproblem:
\begin{equation}\label{eq: fsdef}
\begin{aligned}
&\text{~find~} \bz_{s-1}^\star \text{~such that~} 0\in (F^{s-1} + \partial r)(\bz_{s-1}^\star),\\
&\text{~where~} F^{s-1}(\bz) = F^{s-2}(\bz) + \mu_{s-1} (\bz-\bz_{s-1}) = F^\mu(\bz) + \sum_{i=1}^{s-1} \mu_i(\bz-\bz_i), \text{~if~} s\geq 2,
\end{aligned}
\end{equation}
and for $s=1$, we have that $F^0$ is as defined in \eqref{eq: fdelta} and $\bz_0^\star$ is defined in \eqref{eq: delta_prob}.

That is, at iteration $s$, we form an auxiliary problem which is centered at current iterate $\bz_{s-1}$ and then use a strongly monotone subsolver in \Cref{alg:eg_sc} to solve this problem. The amount of strong monotonicity added is increasing exponentially fast, that is $\mu_s = 2^s\mu$ where $\mu$ is the initial strong monotonicity of $F^0=F^\mu$ defined in \eqref{eq: fdelta}. 

Our main aim is to solve a monotone problem, but similar to \cite{allen2018make}, the approach for the best complexity requires designing the strongly monotone solver with the gradient mapping guarantee (see \Cref{alg:rec_reg} and \Cref{cor:sm}). After this, we will invoke this result on a perturbed problem where the initial strong monotonicity level will depend on the desired accuracy, see \eqref{eq: fdelta}, case 2.

\paragraph{Strongly convex subsolver -- \Cref{alg:eg_sc}.}
The first layer was described in the previous section where we solve a sequence of strongly monotone subproblems. To solve these problems, we use a restarting-based second layer (\Cref{alg:eg_sc}) similar to \cite{allen2018make}. This method restarts \Cref{alg:eg_anchor} with two different step size and inner iteration limit pairs.

\paragraph{Main workhorse: Anchored stochastic extragradient -- \Cref{alg:eg_anchor}.}
Anchoring is mainly used to handle the BG variance assumption (Assumption \ref{asp: 2}). Indeed, if the variance is bounded, then $B=0$ and the algorithm reduces to regular stochastic extragradient method. However, the anchoring allows us to handle nonzero $B$, to cover problems without bounded variance.

We now continue with the summary of two main results that will be developed in the sequel.

\subsection{Complexity for Strongly Monotone Problems}
We now state the main complexity results for solving strongly monotone problems, which correspond to strongly convex-strongly concave instances of \eqref{eq: min-max}. We state this result under bounded variance, since we don't prove it under Blum-Gladyshev variance due to space constraints. The following result is proven in Section \ref{sec: bdd_var_str_monot}. A similar result under BG variance assumption can be proven by using the tools in Section \ref{sec: unbdd_var}. We omit this extension for brevity.
\begin{corollary}\label{cor:sm}
For problem \eqref{eq: mi}, let Assumption \ref{asp: 1} hold and suppose that Assumption \ref{asp: 2} holds with $B=0$. Additionally assume that $F$ is $\mu_F$-strongly monotone. Let $S=\lfloor \log_2(L_F/\mu_F) \rfloor\geq 1$, $T \geq 48SL_F/\mu_F$ and $\eta=\frac{1}{3L_F}$. Then, we have that 
\begin{equation}
\mathbb{E}\|\cG_{\eta, H}(\bz_S)\|\leq\varepsilon,
\end{equation}
where the number of stochastic first-order oracles $T$ is upper bounded by
\begin{equation*}
O\left( \frac{L_F}{\mu_F} \log\left(\frac{L_F}{\mu_F}\right) \log\left( \frac{\mu_F D_\star}{\varepsilon}+e\right) + \frac{G^2 \log^3(L_F/\mu_F)}{\varepsilon^2}\right).
\end{equation*}
where $O$ only suppresses the absolute constants.
\end{corollary}
Even for unconstrained problems, our result extends \citep[Theorem 4.1]{chen2024near} since our algorithmic parameters do not depend on $\|\bz_0-\bx^\star\|$ or a global variance upper bound $\sigma^2$, which were required for running the algorithm in \citep[Theorem 4.1]{chen2024near}. The parameters in our algorithm only depend on the oracle budget $T$, $L_F$ and $\mu_F$.
\subsection{Complexity for Monotone Problems}
By using the result for the strongly monotone case with a perturbed version of the monotone problem, we can obtain the claimed complexity guarantee for solving the original monotone problem. For this we need to connect the solutions in terms of the gradient mapping norm and use the previous result. In particular, we will solve
\begin{equation}\label{eq: sgv4}
\text{find~} \bz^\star_0 \text{~such that~} 0\in H^\mu(\bz^\star_0) := F(\bz^\star_0) + \partial r(\bz^\star_0) + \mu(\bz^\star_0-\zinit),
\end{equation}
where $\mu = \widetilde{O}\left( \frac{1}{\sqrt{T}} \right)$. Since $\mu$ is small depending on the oracle budget $T$, the perturbed problem is sufficiently close to the original problem to yield the claimed complexity result on the gradient mapping. The proof of the following result is the focus of Section \ref{sec: unbdd_var}.
\begin{corollary}
For problem \eqref{eq: mi}, let Assumptions \ref{asp: 1} and \ref{asp: 2} hold. We apply \Cref{alg:rec_reg} to solve \eqref{eq: delta_prob} with $\mu = \frac{32\Lhat_F}{\sqrt{T}}\log_2^{3/2}\left( \frac{\sqrt{T}}{48} \right)$, $T\geq 14175$, $\eta=\frac{1}{3\Lhat_F}$. Then, we have that 
\begin{equation}
\mathbb{E}\|\cG_{\eta, H}(\bz_S)\|= O\left( \frac{\Lhat_FD_\star (\ln \sqrt{T})^{3/2}}{\sqrt{T}} + \frac{(\ln\sqrt{T})^{5/2}(BD_\star+G)}{\sqrt{T}} \right).
\end{equation}
Consequently, to obtain
\begin{equation}
\mathbb{E}\|\cG_{\eta, H}(\bz_S)\|\leq\varepsilon,
\end{equation}
the number of required stochastic first-order oracles $T$ is upper bounded by
\begin{equation*}
O\left( 1+\frac{(L_F^2+B^2)D_\star^2 + G^2}{\varepsilon^2}\ln^5\left(\frac{D_\star(L_F+B)+G}{\varepsilon}+1\right) \right),
\end{equation*}
where we only suppress the absolute constants.
\end{corollary}
A similar discussion as the previous subsection apply here to compare our result with the corresponding result of \cite{chen2024near}. In particular, we not only improve over this result to handle constrained or regularized problems, but also require a weaker variance assumption, and use algorithmic parameters only depending on $L_F, T, B$. An additional point of discussion is the parameter $\mu$, which is independent of $\|\zinit-\bx^\star\|^2$, whereas this quantity was used to set $\mu$ in the work of  \cite{chen2024near}. To transfer our guarantees to a guarantee on the stronger \emph{tangent residual} (see its definition, e.g., in \cite{cai2022tight}), one can apply a postprocessing step similar to \cite[Appendix C.3]{cai2024variance}.

\section{Analysis for Subsolvers}
\subsection{Layer 1: Anchored Stochastic Extragradient}
We now analyze Alg. \ref{alg:eg_anchor} which is the main workhorse of the construction. It is an extragradient-based method for solving a strongly monotone inclusion problem. The addition on top of the standard extragradient is the Halpern-type anchoring step \cite{halpern1967fixed}. 

Even though we are not aware of this algorithm being proposed or analyzed before; in spirit, it can be considered as a stochastic version of regular anchored extragradient (see \cite{yoon2021accelerated} and \citep[Algorithm 3]{kovalev2022first}) or a single-sample version of \cite{lee2021fast}, or non-variance reduced version of \citep[Section 4.3]{alacaoglu2025towards}. Our choice for building on extragradient is purely for presentation purposes, as similar bounds can be proven for anchored versions of other algorithms.

All these works above focused on the non-strongly monotone case. However, our subproblems are strongly monotone, so we analyze for this case. We will call this method and its complexity result for a series of subproblems, so we introduce the problem:
\begin{align}\label{eq: prob_innermost1}
\text{find~}\bx^\star_A \text{~such that~} 0\in (A+\partial r)(\bx^\star_A),
\end{align}
when $A$ is strongly monotone. At each step $s$ of \Cref{alg:rec_reg}, we will change $A$ depending on $F^{s-1}$, see \eqref{eq: fsdef}.

For convenience, let us introduce the assumptions to be used in this and next subsection.
\begin{assumption}\label{asp: inner}
For problem \eqref{eq: prob_innermost1}, let $A$ be monotone, $L_A$-Lipschitz and $\mu_A$-strongly monotone. Assume that $r$ is proper, convex, closed. There exists $\bx^\star_A$ such that $0\in(A+\partial r)(\bx^\star_A)$. 

We have access to an unbiased estimator of $A$ such that $\mathbb{E}[A_\xi(\bx)] = A(\bx)$ and
\begin{equation*}
\mathbb{E} \| A_\xi(\bx) - A(\bx) \|^2 \leq B_A^2 \| \bx-\zinit\|^2 + G_A^2,
\end{equation*}
where $\zinit$ is a globally fixed iterate.
\end{assumption}
To see the connection between the parameters $B_A, G_A$ of the variance assumption in this section and the variance assumption made for the global problem (Assumption \ref{asp: 2}), see Fact \ref{fact:struc}.

The following lemma, when $\beta=0$ is completely classical, see \cite{juditsky2011solving} for an early reference. We provide a generalization here with the anchoring step. We build on \cite{neu2024dealing,alacaoglu2025towards} who analyzed similar algorithms in different contexts, without strong monotonicity. The main idea is to use the anchoring with weight $\beta$ to absorb the contribution coming from the iterate-dependent $B_A$ term in the BG assumption (Assumption \ref{asp: inner}), which allows relaxing the bounded variance assumption.
\begin{lemma}\label{lem: eg}
For problem \eqref{eq: prob_innermost1}, let Assumption \ref{asp: inner} hold.
When we run Algorithm \ref{alg:eg_anchor} for $K$ iterations, with parameters $\alpha \leq  \frac{1}{2(L_A+B_A)}$, $\beta=3\alpha^2B_A^2$, we obtain
\begin{equation*}
\mathbb{E} \| \hat \bx_K^{s, n} - \bx^\star_A \|^2 \leq \frac{1}{2\alpha \mu_A K} \| \bx_0^{s, n}-\bx^\star_A\|^2 + \frac{3\alpha}{\mu_A}\left(B_A^2\| \zinit-\bx^\star_A\|^2 + \frac{7G_A^2}{18} \right).
\end{equation*}
\end{lemma}
\begin{remark}
In this bound, the additional term $B^2_A\| \zinit-\bx^\star_A\|^2$ coming from the relaxed variance requirement seems not problematic since $\bx^\star_A$ is fixed. However, we will later invoke this result with a dynamic set of problems where the solution $\bx^\star_A$ will be changing at each step. As a result, we cannot treat this term as a constant. We will have to craft a specialized analysis to handle this additional term since we do not have a global upper bound for the solutions of inner problems. This will be handled in Section \ref{sec: unbdd_var}.
\end{remark}
\begin{proof}[Proof of \Cref{lem: eg}]
For a lighter notation, we suppress the superscript $s, n$ from the iterates, as in the algorithm.

By the update rules of the iterates $\bx_{k+1/2}, \bx_{k+1}$ in Algorithm \ref{alg:eg_anchor}, and \eqref{eq: prox_ineq}, we have
\begin{equation*}
\begin{aligned}
\langle \bx_{k+1/2} - \bar \bx_k + \alpha A_{\xi_k}(\bx_k), \bx_{k+1} - \bx_{k+1/2} \rangle \geq \alpha (r(\bx_{k+1/2}) - r(\bx_{k+1})),\\
\langle \bx_{k+1} - \bar \bx_k + \alpha A_{\xi_{k+1/2}}(\bx_{k+1/2}), \bx^\star_A - \bx_{k+1} \rangle \geq \alpha (r(\bx_{k+1}) - r(\bx^\star_A)).
\end{aligned}
\end{equation*}
We sum up these inequalities and rearrange to obtain
\begin{equation}\label{eq: anch_step1}
\begin{aligned}
&\alpha(r(\bx_{k+1/2}) - r(\bx^\star_A)+\langle A_{\xi_{k+1/2}}(\bx_{k+1/2}), \bx_{k+1/2} - \bx^\star_A \rangle) \\
&\leq  \langle \bx_{k+1/2} - \bar \bx_k, \bx_{k+1} - \bx_{k+1/2} \rangle + \langle \bx_{k+1} - \bar \bx_k, \bx^\star_A - \bx_{k+1} \rangle \\
&\quad + \alpha \langle A_{\xi_k}(\bx_k) - A_{\xi_{k+1/2}}(\bx_{k+1/2}), \bx_{k+1}-\bx_{k+1/2} \rangle.
\end{aligned}
\end{equation}
We estimate the first two terms on the right-hand side. Applying $\|\ba+\bb\|^2 = \| \ba\|^2 + \|\bb\|^2 + 2\langle \ba,\bb \rangle$ gives
\begin{align*}
2\langle \bx_{k+1/2} - \bar \bx_k, \bx_{k+1} - \bx_{k+1/2} \rangle &= \| \bx_{k+1}-\bar\bx_k\|^2 - \| \bx_{k+1/2} - \bar\bx_k\|^2 - \| \bx_{k+1} - \bx_{k+1/2}\|^2,\\
2\langle \bx_{k+1} - \bar \bx_k, \bx^\star_A-\bx_{k+1}  \rangle &= \| \bx^\star_A-\bar\bx_k\|^2 - \| \bx_{k+1} - \bar\bx_k\|^2 - \| \bx^\star_A - \bx_{k+1}\|^2.
\end{align*}
Plugging in these identities to \eqref{eq: anch_step1}, after multiplying both sides by $2$, and taking expectation give
\begin{equation}\label{eq: anch_step2}
\begin{aligned}
&2\alpha\mathbb{E}[r(\bx_{k+1/2}) - r(\bx^\star_A)+\langle A_{\xi_{k+1/2}}(\bx_{k+1/2}), \bx_{k+1/2} - \bx^\star_A\rangle ] \\
&\leq \mathbb{E} \left[ \| \bx^\star_A-\bar\bx_k\|^2 - \| \bx^\star_A-\bx_{k+1}\|^2 - \| \bx_{k+1} - \bx_{k+1/2} \|^2 - \| \bx_{k+1/2}-\bar\bx_k\|^2\right] \\
&\quad + 2\alpha\mathbb{E} \langle A_{\xi_k}(\bx_k) - A_{\xi_{k+1/2}}(\bx_{k+1/2}), \bx_{k+1}-\bx_{k+1/2} \rangle.
\end{aligned}
\end{equation}
We now estimate the inner product on the right-hand side by
\begin{align}
&2\alpha\mathbb{E} \langle A_{\xi_k}(\bx_k) - A_{\xi_{k+1/2}}(\bx_{k+1/2}), \bx_{k+1} - \bx_{k+1/2} \rangle \notag \\
 &\leq \mathbb{E}\left[\alpha^2\|A_{\xi_k}(\bx_k) - A_{\xi_{k+1/2}}(\bx_{k+1/2})\|^2 + \|\bx_{k+1} - \bx_{k+1/2}\|^2 \right] \notag \\
&\leq \alpha^2 \mathbb{E} \left[ \frac43\| A_{\xi_k}(\bx_k)-A(\bx_k)  \|^2 + \| A(\bx_{k+1/2}) - A_{\xi_{k+1/2}}(\bx_{k+1/2}) \|^2 + 4L_A^2  \| \bx_{k+1/2}-\bx_k  \|^2 \right] \notag \\
&\quad +\E\|\bx_{k+1} - \bx_{k+1/2}\|^2.\label{eq: sof4}
\end{align}
Here, the first step is by Young's inequality and the second step used $L_A$-Lipschitzness of $A$ after applying the estimation
\begin{align*}
&\mathbb{E}\left[ \|A_{\xi_k}(\bx_k) - A_{\xi_{k+1/2}}(\bx_{k+1/2})\|^2 \right] \\
&= \mathbb{E} \left[ \| A_{\xi_k}(\bx_k) - A(\bx_{k+1/2})\|^2 + \| A(\bx_{k+1/2}) - A_{\xi_{k+1/2}}(\bx_{k+1/2})\|^2 \right] \\
&\leq \mathbb{E} \left[ \frac{4}{3}\| A_{\xi_k}(\bx_k) - A(\bx_k)\|^2 + 4\| A(\bx_k) - A(\bx_{k+1/2})\|^2 + \|A(\bx_{k+1/2}) - A_{\xi_{k+1/2}}(\bx_{k+1/2})\|^2 \right],
\end{align*}
where the first line is by tower property, since $\mathbb{E}_k[A(\bx_{k+1/2}) - A_{\xi_{k+1/2}}(\bx_{k+1/2})] = 0$ and $A_{\xi_k}(\bx_k) - A(\bx_{k+1/2})$ is measurable under the conditioning of $\mathbb{E}_k$.

For the inner product on the left-hand side of \eqref{eq: anch_step2}, we use strong monotonicity of $A$ to derive
\begin{align}
\mathbb{E}\langle A_{\xi_{k+1/2}}(\bx_{k+1/2}), \bx_{k+1/2} - \bx^\star_A\rangle &= \mathbb{E}\langle {A}(\bx_{k+1/2}), \bx_{k+1/2}-\bx^\star_A\rangle\notag \\
&\geq \mathbb{E}[\langle {A}(\bx^\star_A), \bx_{k+1/2}-\bx^\star_A\rangle + \mu_A \| \bx^\star_A-\bx_{k+1/2}\|^2]\notag \\
&\geq \mathbb{E}[r(\bx^\star_A) - r(\bx_{k+1/2}) +\mu_A \| \bx^\star_A-\bx_{k+1/2}\|^2],\label{eq: str_mont}
\end{align}
where the first identity used the tower property, $\mathbb{E}_k[A_{\xi_{k+1/2}}(\bx_{k+1/2})] = A(\bx_{k+1/2})$, and that $\bx_{k+1/2}-\bx^\star_A$ is measurable under the conditioning of $\mathbb{E}_k$. The third step used convexity of $r$ as well as $-A(\bx^\star_A) \in \partial r(\bx^\star_A)$ by the definition of the solution $\bx^\star_A$ in \eqref{eq: prob_innermost1}.
Using \eqref{eq: sof4} and \eqref{eq: str_mont} in \eqref{eq: anch_step2} gives
\begin{align}
2\alpha\mu_A \mathbb{E} \| \bx^\star_A - \bx_{k+1/2}\|^2 &\leq \mathbb{E} \left[ \| \bx^\star_A-\bar\bx_k\|^2 - \| \bx^\star_A-\bx_{k+1}\|^2  - \| \bar\bx_k-\bx_{k+1/2}\|^2\right] \notag \\
&\quad + \alpha^2\mathbb{E}\left[ \frac43\| A(\bx_k) - A_{\xi_k}(\bx_k)\|^2 + \| A(\bx_{k+1/2}) - A_{\xi_{k+1/2}}(\bx_{k+1/2})\|^2 \right]\notag \\
&\quad  + 4\alpha^2L_A^2 \mathbb{E}\| \bx_k-\bx_{k+1/2}\|^2.\label{eq: sob5}
\end{align}
On the one hand, by the definition of $\bar\bx_k$ in Algorithm \ref{alg:eg_anchor}, we have
\begin{align*}
\|\bx^\star_A - \bar\bx_k\|^2 - \|\bar\bx_k-\bx_{k+1/2}\|^2 &= \beta(\|\bx^\star_A - \zinit\|^2 - \|\zinit-\bx_{k+1/2}\|^2)+ (1-\beta)(\|\bx^\star_A-\bx_k\|^2 - \|\bx_k-\bx_{k+1/2}\|^2)
\end{align*}
and, by Young's inequality,
\begin{equation*}
-\beta \|\bx^\star_A-\bx_{k}\|^2 \leq -\frac{\beta}{2}  \| \zinit-\bx_{k}\|^2 + \beta \| \zinit-\bx^\star_A\|^2.
\end{equation*}
On the other hand, to handle the terms in the second line of \eqref{eq: sob5}, we have by Assumption \ref{asp: inner} that
\begin{align*}
&\alpha^2\mathbb{E}\left[ \frac{4}{3}\| A(\bx_k) - A_{\xi_k}(\bx_k)\|^2 + \| A(\bx_{k+1/2}) - A_{\xi_{k+1/2}}(\bx_{k+1/2})\|^2 \right] \\
&\quad\leq \alpha^2B_A^2\mathbb{E}\left[ \frac{4}{3}\| \bx_k-\zinit\|^2 + \| \bx_{k+1/2} - \zinit\|^2 \right] + \frac73\alpha^2G_A^2.
\end{align*}
Combining the last three estimates in \eqref{eq: sob5} gives
\begin{align}
2\alpha\mu_A \mathbb{E} \| \bx^\star_A - \bx_{k+1/2}\|^2 &\leq \mathbb{E} \left[\| \bx^\star_A-\bx_k\|^2 - \| \bx^\star_A-\bx_{k+1}\|^2+2\beta\|\zinit-\bx^\star_A\|^2 \right] +\frac73\alpha^2G_A^2\notag \\
&\quad + \left( \frac{4\alpha^2B_A^2}{3} - \frac{\beta}{2} \right)\mathbb{E} \| \zinit-\bx_k\|^2 + (\alpha^2B_A^2-\beta)\E\|\zinit-\bx_{k+1/2}\|^2 \notag \\
&\quad  + \left(4\alpha^2L_A^2-(1-\beta) \right) \mathbb{E}\| \bx_k-\bx_{k+1/2}\|^2.
\end{align}
By the requirements on $\alpha$ and $\beta$, we have that the last three terms on the right-hand side are nonpositive.

Dividing both sides by $2\alpha\mu_A$, summing for $k=0,\dots, K-1$, dividing by $K$ and using the definition of $\hat \bx_K^{s, n}$  from \Cref{alg:eg_anchor} gives the result.
\end{proof}

\subsection{Layer 2: Restarting for Better Complexity}
The following result analyzes the restart-based instantiations of the inner solver and the argument is the same as \cite{allen2018make} (see also \cite{hazan2014beyond}), invoked with slightly different parameters and constants.
\begin{lemma}\label{lem: level2}
For problem \eqref{eq: prob_innermost1}, let Assumption \ref{asp: inner} hold. When we run Algorithm \ref{alg:eg_sc} with budget $T\geq 32\Lhat_A/\mu_A$ for the number of oracles and initial point $\by_0$, we have
\begin{equation*}
\mathbb{E} \| \by_{N+M}^s - \bx^\star_A \|^2 \leq \left(\frac{1}{2}\right)^{\frac{T}{8\Lhat_A/\mu_A}} \|\by_0^s-\bx^\star_A\|^2 + \frac{192B_A^2\|\zinit-\bx^\star_A\|^2 +75G_A^2}{T\mu_A^2},
\end{equation*}
where $\Lhat_A = L_A+B_A$.
The total number of oracle calls is upper bounded by $T$.
\end{lemma}
\begin{proof}
For a lighter notation, we omit the superscript $s$ as in the algorithm and also suppress the subscript for $\mu_A, \widehat{L}_A, B_A, G_A$ throughout the proof. Let us use the notations
\begin{align*}
\Ghat^2 = B^2\|\zinit-\bx^\star_A\|^2 + \frac{7G^2}{18}, \text{~~~and~~~} \Lhat = L+B.
\end{align*}
We invoke \Cref{lem: eg} with $\alpha=\frac{1}{2\Lhat}$ and $K=\frac{2\Lhat}{\mu}$. Since $\alpha \mu K = 1$ and $\frac{3\alpha}{\mu} = \frac{3}{2\Lhat\mu}$, we have for each $n=1,\dots, N$ (where the initial point to \Cref{alg:eg_anchor} is $\by_{n-1}$):
\begin{align*}
\mathbb{E} \| \by_n - \bx^\star_A \|^2 &\leq \frac{1}{2} \E\| \by_{n-1}-\bx^\star_A\|^2 + \frac{3}{2\Lhat\mu}\Ghat^2 \\
&\leq \frac{1}{2^n} \|\by_0-\bx^\star_A\|^2 + \frac{3\Ghat^2}{2\Lhat\mu}\sum_{i=0}^{n-1} \frac{1}{2^i}.
\end{align*}
At $n=N$, this gives us
\begin{equation}\label{eq: ynbd}
\mathbb{E} \| \by_N - \bx^\star_A\|^2 \leq \frac{1}{2^N} \| \by_0-\bx^\star_A\|^2 + \frac{3\Ghat^2}{\Lhat \mu}(1-2^{-N}),
\end{equation}
where the last term is by summing the geometric series.

We continue to analyze the iterations of Algorithm \ref{alg:eg_sc} when $n=N+1$ to $N+M$. By the setting of $\alpha$ and $K$ in this case, we have $\alpha \mu K= 2$ and $\frac{3\alpha}{\mu} = \frac{3}{2^{n-N}\Lhat\mu}$. We invoke \Cref{lem: eg} at $n=N+M$ and then unroll until $n=N+1$ to obtain
\begin{align*}
\mathbb{E} \| \by_{N+M} - \bx^\star_A\|^2 &\leq \frac{1}{4} \mathbb{E} \| \by_{N+M-1}-\bx^\star_A\|^2 + \frac{3\Ghat^2}{2^{M}\Lhat\mu} \\
&\leq \frac{1}{4^M}\mathbb{E}\|\by_N-\bx^\star_A\|^2 + \frac{3\Ghat^2}{2^M\Lhat\mu}\sum_{i=0}^{M-1} \frac{1}{2^i} \\
&\leq \frac{1}{4^M}\mathbb{E}\|\by_N-\bx^\star_A\|^2 + \frac{3\Ghat^2}{2^{M-1}\Lhat\mu}(1-2^{-M}),
\end{align*}
where the last line is by summing the geometric series.
On the last inequality, we plug in the bound of $\mathbb{E}\| \by_N-\bx^\star_A\|^2$ from \eqref{eq: ynbd} to deduce
\begin{align*}
\mathbb{E} \| \by_{N+M} - \bx^\star_A\|^2 \leq \frac{1}{2^{N+2M}} \|\by_0-\bx^\star_A\|^2 + \frac{6\Ghat^2}{2^M\Lhat\mu}.
\end{align*}
Since $N\geq\frac{T}{8\Lhat/\mu}-1$, $2^M \geq  \frac{T}{32\Lhat/\mu}$, and $M\geq 1$, we have
\begin{align*}
\mathbb{E} \| \by_{N+M} - \bx^\star_A\|^2 \leq \left(\frac{1}{2}\right)^{\frac{T}{8\Lhat/\mu}} \|\by_0-\bx^\star_A\|^2 + \frac{192\Ghat^2}{T\mu^2}.
\end{align*}
Moreover, the total number of calls to $A_\xi$ is less than
\begin{equation*}
2\times \frac{2\Lhat}{\mu} \times N + \sum_{m=1}^M 2\times 2^{m+1}\times \frac{\Lhat}{\mu} \leq \frac{T}{2} +  2^M\times \frac{8\Lhat}{\mu}\leq T,
\end{equation*}
since $N\leq \frac{T}{8\Lhat/\mu}$ and $2^M \leq \frac{T}{16\Lhat/\mu}$.
\end{proof}

\section{Complexity Analysis with Bounded Variance}
We now analyze \Cref{alg:rec_reg}. This method calls \Cref{alg:eg_sc} repeatedly for $S$ times where each step approximates the problem in \eqref{eq: fsdef}.

For analyzing this scheme, we will invoke \Cref{lem: level2}, with $\bx_A^\star \leftarrow \bz^{\star}_{s-1}$ and $\by_0 \leftarrow \bz_{s-1}$. 
Since the last term on the right-hand side of the result of \Cref{lem: level2} has a term depending on $B^2\|\zinit - \bz_{s-1}^\star\|^2$, we need a dedicated analysis to control this term since such a term did not exist in prior analyses that required a uniformly bounded variance \citep{chen2024near, allen2018make}. For simplicity, we will first analyze the case $B=0$ which corresponds to bounded variance and then we will show how to handle the $B\neq 0$ case in the next section.

The two results below will be used for both monotone and strongly monotone cases. As a result, they only use that $F^\mu$ is $\mu$-strongly monotone which is true in both cases, due to the definition \eqref{eq: fdelta}.
\begin{lemma}\label{lem: iter_dif}
For problem \eqref{eq: mi}, let Assumption \ref{asp: 1} hold and suppose that \Cref{asp: 2} holds with $B=0$. We apply \Cref{alg:rec_reg} to solve \eqref{eq: delta_prob} with budget $T\geq 48\kappa S$ for number of oracles where $\kappa=\frac{L_F}{\mu}$, $S=\left\lfloor  \log_2\kappa \right\rfloor \geq 1$ and initial point $\zinit$. We have for all $s=1,\dots, S$ that
\begin{align*}
\mu_s^2 \mathbb{E} \| \bz_{s} - \bz_{s-1}^\star\|^2 \leq \left( \frac{1}{2} \right)^{\frac{Ts}{24\kappa S}} \mu^2 \|\zinit-\bz^\star_0\|^2 + \frac{400SG^2}{T},
\end{align*}
and the points $\bz_{s}^\star$ are as defined in \eqref{eq: fsdef}.
\end{lemma}
\begin{proof}
At iteration $s\geq1$, we will invoke \Cref{lem: level2} with 
\begin{equation}
\by_0^s\leftarrow \bz_{s-1},~~~ \bx^\star_A\leftarrow \bz_{s-1}^{\star},~~~ A\leftarrow F^{s-1},~~~ T\leftarrow T/S,
\end{equation}
where $F^0 = F^\mu$ as defined in \eqref{eq: fdelta}.

With these settings, we can verify that the requirements of \Cref{asp: inner} hold. Then, the structural constants in the setting of \Cref{lem: level2} become
\begin{equation}
\mu_A\leftarrow (2^{s}-1)\mu, ~~~ \widehat{L}_A \equiv L_A \leftarrow L_F+(2^s-1)\mu, ~~~\frac{\Lhat_A}{\mu_A} = \frac{L_F+(2^s-1)\mu}{(2^s-1)\mu} \leq \frac{3\kappa}{2}, ~~~B_A \leftarrow 0, ~~~ G_A \leftarrow G,
\end{equation}
where $\kappa = \frac{L_F}{\mu}\geq2$,
see \Cref{fact:struc} and $\mu$ is the strong monotonicity constant of $F^\mu$. The last two conclusions also follow from \Cref{fact:struc}. 

As a result, from \Cref{lem: level2}, we have for $s \geq 1$:
\begin{align}\label{eq: seq1}
\mathbb{E} \| \bz_s - \bz_{s-1}^\star \|^2 &\leq \left(\frac{1}{2} \right)^{\frac{T}{12S\kappa}} \mathbb{E} \| \bz_{s-1} - \bz_{s-1}^\star \|^2 + \frac{75S G^2}{T\mu_{s-1}^2},
\end{align}
since $\mu_{s-1} \leq \mu_A$.

Particularly, for $s=1$, we have (recalling $\mu_0 =\mu$)
\begin{equation}\label{eq: roll_s_1case}
\mu_1^2\mathbb{E} \| \bz_1-\bz_0^\star\|^2 \leq \left(\frac{1}{2}\right)^{\frac{T}{24S\kappa}} \mu^2\| \zinit - \bz_0^\star\|^2 + \frac{300SG^2}{T},
\end{equation}
where we used $4\left( \frac{1}{2} \right)^{\frac{T}{12\kappa S}} \leq \left( \frac{1}{2} \right)^{\frac{T}{24\kappa S}}$ due to $T \geq 48\kappa S$.

Using \Cref{eq: nonexp1} and $\mu_s^2 = 4\mu_{s-1}^2$, this implies that for $s\geq 2$:
\begin{align*}
\mu_s^2 \mathbb{E} \| \bz_s - \bz_{s-1}^\star \|^2 &\leq \left(\frac{1}{2} \right)^{\frac{T}{12S\kappa}} 4\mu_{s-1}^2\mathbb{E} \| \bz_{s-1} - \bz_{s-2}^\star \|^2 + \frac{300S G^2}{T} \\
&\leq \left(\frac{1}{2} \right)^{\frac{T}{24S\kappa}} \mu_{s-1}^2\mathbb{E} \| \bz_{s-1} - \bz_{s-2}^\star \|^2 + \frac{300SG^2}{T},
\end{align*}
where the last step used $T \geq 48S\kappa$ to get $\frac{T}{12S\kappa} -2\geq \frac{T}{24S\kappa}$.
By unrolling this inequality and summing the geometric series for the last term, we get for $s\geq 2$:
\begin{align*}
\mu_s^2\mathbb{E} \| \bz_s - \bz_{s-1}^\star \|^2 &\leq \left(\frac{1}{2} \right)^{\frac{(s-1)T}{24S\kappa}} \mu_1^2\mathbb{E} \| \bz_{1} - \bz_{0}^\star \|^2 +  \frac{300SG^2}{T}\sum_{i=0}^{s-2}\left(\frac{1}{2}\right)^{\frac{iT}{24\kappa S}} \\
&\leq \left(\frac{1}{2} \right)^{\frac{sT}{24S\kappa}}  \mu^2 \| \zinit-\bz_0^\star\|^2 +  \frac{300SG^2}{T}\sum_{i=0}^{s-1}\left(\frac{1}{2}\right)^{\frac{iT}{24\kappa S}},
\end{align*}
where the second inequality used \eqref{eq: roll_s_1case}. Calculating the sum of geometric series and using $T \geq 48\kappa S$ gives the result.
\end{proof}
By using \Cref{th: it_dif_to_gradmapmu}, we can now convert this result to a guarantee on the gradient mapping norm for the operator $H^\mu$ given in \eqref{eq: delta_prob}. In view of its definition, this will directly give us the required guarantee for the strongly monotone setting. In the monotone setting, we will select $\mu$ accordingly to get a guarantee for the original problem in \eqref{eq: mi}.
\begin{lemma}\label{th: hdelta}
For problem \eqref{eq: mi}, let Assumption \ref{asp: 1} hold and suppose that \Cref{asp: 2} holds with $B=0$. We apply \Cref{alg:rec_reg} to solve \eqref{eq: delta_prob} with $\eta= \frac{1}{3L_F}$, oracle budget $T \geq 48SL_F/\mu$ where $S=\left\lfloor \log_2\frac{L_F}{\mu} \right\rfloor \geq 1$ for the number of oracles and initial point $\zinit$. Then, we have that
\begin{align*}
\mathbb{E} \| \cG_{\eta, H^\mu}(\bz_S)\| = O\left( \left(\frac{1}{2}\right)^{\frac{T}{48SL_F/\mu}} \mu \| \zinit - \bz_0^\star\| + \frac{S^{3/2}G}{\sqrt{T}} \right),
\end{align*}
where $O$ only suppresses the absolute constants.
\end{lemma}
\begin{proof}
In view of the result of \Cref{lem: iter_dif}, we invoke \Cref{th: it_dif_to_gradmapmu} with $C\leftarrow 400$ and $\Ghat \leftarrow G$ to obtain the claim.
\end{proof}

\subsection{Strongly Monotone Case}\label{sec: bdd_var_str_monot}
In the strongly monotone case, we will use the first case in the definition of $F^\mu$ in \eqref{eq: fdelta}. This reduces $H^\mu$ to the original problem and we have the following theorem as the direct corollary of \Cref{th: hdelta}.
\begin{theorem}\label{th: scsc}
For problem \eqref{eq: mi}, let Assumption \ref{asp: 1} hold and suppose that Assumption \ref{asp: 2} holds with $B=0$. Additionally assume that $F$ is $\mu_F$ strongly monotone for $\mu_F>0$. We apply \Cref{alg:rec_reg} to solve \eqref{eq: delta_prob} with $\eta=\frac{1}{3L_F}$, $S=\lfloor \log_2(L_F/\mu_F) \rfloor\geq 1$ and $T \geq 48SL_F/\mu_F$, and initial point $\zinit$.
Then, we have that 
\begin{align*}
\mathbb{E} \| \cG_{\eta, H}(\bz_S)\| = O\left( \left(\frac{1}{2}\right)^{\frac{T}{48SL_F/\mu_F}} \mu_FD_\star + \frac{S^{3/2}G}{\sqrt{T}} \right),
\end{align*}
where $O$ only suppresses the absolute constants.
\end{theorem}
\begin{proof}
Since $F$ is strongly monotone, in view of \eqref{eq: fdelta}, we have $F^\mu = F$, $\mu=\mu_F$, and $\bz_0^\star = \bx^\star$. Then, we invoke \Cref{th: hdelta} with these settings and use the definition of $D_\star$ from \eqref{eq:dstar} to derive the bound.
\end{proof}
The next corollary then immediately follows by finding the value of $T$ to make the right-hand side of \Cref{th: scsc} less than $\varepsilon$.
\begin{corollary}\label{cor: str_monot}
Under the same setup as \Cref{th: scsc}, we have
\begin{equation}
\mathbb{E}\|\cG_{\eta, H}(\bz_S)\|\leq\varepsilon,
\end{equation}
where the number of stochastic first-order oracles $T$ is upper bounded by
\begin{equation*}
O\left( \frac{L_F}{\mu_F} \log\left(\frac{L_F}{\mu_F}\right) \log\left( \frac{\mu_F D_\star}{\varepsilon}+e\right) + \frac{G^2 \log^3(L_F/\mu_F)}{\varepsilon^2}\right).
\end{equation*}
\end{corollary}
Our algorithmic parameters do not depend on $\|\bz_0-\bx^\star\|$ or a global variance upper bound $\sigma^2$, which were required for running the algorithm in \citep[Theorem 4.1]{chen2024near}. The parameters in our algorithm only depend on the oracle budget $T$, $L_F$ and $\mu_F$, similar to the case of \cite{allen2018make} in the minimization case and the bound for the gradient mapping in \Cref{th: scsc} follows. Converting this to a complexity bound gives Corollary \ref{cor: str_monot}.
\subsection{Monotone Case}\label{subsec:monot_bddvar}
In view of the definitions in \eqref{eq: delta_prob} and \eqref{eq: fdelta}, we are in the second case of \eqref{eq: fdelta}. Let us recall that \Cref{th: hdelta} gives a bound for solving the perturbed problem, controlled by $\mu$. By controlling $\mu$ and the distance between gradient mapping at perturbed problem and the original problem, we will derive a guarantee for the gradient mapping for solving the original problem in \eqref{eq: mi}.
\begin{theorem}[Monotone $F$]\label{th: cc}
For problem \eqref{eq: mi}, let Assumption \ref{asp: 1} hold and suppose that \Cref{asp: 2} holds with $B=0$. We apply \Cref{alg:rec_reg} to solve \eqref{eq: delta_prob} with any $L_F/2\geq \mu>0$, $\eta=\frac{1}{3L_F}$, $S=\lfloor \log_2(L_F/\mu) \rfloor\geq 1$ and $T \geq 48SL_F/\mu$, and initial point $\zinit$.
Then, we have that 
\begin{equation*}
\mathbb{E} \| \cG_{\eta, H}(\bz_S) \| = O\left( \mu\|\bx^\star - \zinit\| + \left(\frac{1}{2}\right)^{\frac{T}{48SL_F/\mu}} \mu\|\zinit-\bx^\star\| + \frac{S^{3/2}G}{\sqrt{T}} \right).
\end{equation*}
\end{theorem}
\begin{proof}
Applying \Cref{lem: reg_to_unreg} yields the inequality
\begin{align}\label{eq: sso4}
\| \cG_{\eta, H}(\bz_S)\| \leq 2.5 \sum_{j=1}^{S} \mu_j \| \bz_j - \bz_{j-1}^\star\| + 9L_F \| \bz_S - \bz_{S-1}^\star\| + \mu \|\zinit-\bx^\star\|.
\end{align}
Since \Cref{lem: iter_dif} holds in this case, on \eqref{eq: sso4}, we use \eqref{eq: sll5} to bound the first term and \eqref{eq: bdd_lastterm} the second term on the right-hand side and use \eqref{eq:init_dist_sol} to obtain the claimed bound.
\end{proof}
\begin{corollary}
Under the same setup of \Cref{th: cc}, let 
$T\geq 2200$
 and 
$\mu= \frac{96L_F\log_2 T}{T}$. We apply \Cref{alg:rec_reg} to solve \eqref{eq: delta_prob} with  $\eta=\frac{1}{3L_F}$, $S=\lfloor \log_2(L_F/\mu) \rfloor$, and initial point $\zinit$. Then, we have that
\begin{equation}\label{eq: bdd_var_ratebd}
\mathbb{E}\|\cG_{\eta, H}(\bz_S)\|= O\left( \frac{L_FD_\star\log_2T}{T} + \frac{G(\log_2T)^{3/2}}{\sqrt{T}} \right).
\end{equation}
Consequently, to obtain
\begin{equation}
\mathbb{E}\|\cG_{\eta, H}(\bz_S)\|\leq\varepsilon,
\end{equation}
the required number of stochastic first-order oracles $T$ is upper bounded by
\begin{equation*}
O\left( 1+ \frac{L_F D_\star}{\varepsilon}\ln\left(\frac{L_FD_\star}{\varepsilon}+1\right) + \frac{G^2}{\varepsilon^2} \ln^3\left(\frac{G}{\varepsilon}+1\right)\right),
\end{equation*}
where we only suppress the absolute constants.
\end{corollary}
\begin{proof}
In this case, we start from the result of \Cref{lem: reg_to_unreg}, after plugging in from \eqref{eq: sll5} and \eqref{eq: bdd_lastterm} to derive
\begin{align}\label{eq: vbn5}
\E\|\cG_{\eta, H}(\bz_S) \| \leq \mu D_\star + 5\left(\frac{1}{2}\right)^{\frac{T}{48\kappa S}}\mu D_\star + \frac{50S^{3/2}G}{\sqrt{T}} 
+ 18\bigg( \left( \frac{1}{2} \right)^{\frac{T}{48\kappa}} \mu D_\star + \frac{20\sqrt{S}G}{\sqrt{T}} \bigg),
\end{align}
where we also used \Cref{eq: nonexp1} which gives
$\| \zinit - \bz^\star_0 \| \leq \|\zinit - \bx^\star\|=D_\star$. Plugging in the definition of $\mu$ gives \eqref{eq: bdd_var_ratebd}. From this bound, it should be clear that this choice of $\mu$ will give a complexity result scaling as $\varepsilon^{-2}$ up to logarithmic terms. We now flesh out the explicit dependencies.

First, we verify that the requirement on $T$ in \Cref{th: cc} is satisfied. That is, we show that $T \geq 48\kappa S$. By the definitions of $\kappa = \frac{L_F}{\mu} = \frac{T}{96\log_2T }\leq T$ and $S=\left\lfloor \log_2\frac{L_F}{\mu} \right\rfloor \leq \log_2\frac{L_F}{\mu} \leq \log_2 T$, we have
\begin{align*}
48\kappa S = \frac{T}{2\log_2T}S \leq T.
\end{align*}
Moreover, by the definition of $\kappa$, we also have that $\kappa \geq 2$ since $T\geq 2200$.

Because of this requirement, for making the first, second and fourth terms in \eqref{eq: vbn5} less than $\varepsilon$, we need $\mu D_\star \leq \varepsilon/9$. This is true when
\begin{align*}
\frac{96L_FD_\star\log_2T}{T} \leq\frac{\varepsilon}{9},
\end{align*}
which is true for $T = \Theta\left( \frac{L_FD_\star}{\varepsilon}\ln\left(\frac{L_FD_\star}{\varepsilon}+1\right) \right)$ because of \Cref{lem:numer}\emph{(ii)} with $a=96L_FD_\star$, $b=1$, $c=1$. $\varepsilon'=\varepsilon/9$.

We finally calculate the order of $T$ to make the third and fifth terms on the right-hand side of \eqref{eq: vbn5} less than $\varepsilon$. Since the third term dominates the fifth up to absolute constants, we calculate $T$ to make
\begin{align*}
\frac{50S^{3/2}G}{\sqrt{T}} \leq \varepsilon,
\end{align*}
where $S = \left\lfloor \log_2\frac{L_F}{\mu} \right\rfloor \leq \log_2\frac{L_F}{\mu} \leq \log_2T$. As a result, a sufficient condition is
\begin{align*}
\frac{2500G^2(\log_2T)^{3}}{T} \leq \varepsilon^2,
\end{align*}
which is true for $T=\Theta\left( \frac{G^2}{\varepsilon^2}\ln^3\left(\frac{G}{\varepsilon}+1\right) \right)$ because of \Cref{lem:numer}\emph{(ii)} with $a=2500G^2$, $b=1$, $c=3$. $\varepsilon'=\varepsilon^2$.

Combining the two bounds for $T$ gives the assertion.
\end{proof}

\section{Complexity Analysis without Bounded Variance}\label{sec: unbdd_var}
For brevity, we will only analyze the monotone case, the strongly monotone case can be analyzed similarly by replacing the $\mu$ value by the strong monotonicity of the operator.

We now derive the result corresponding to \Cref{lem: iter_dif} when the bounded variance assumption is lifted.
Even for unconstrained problems, our result extends \citep[Theorem 4.1]{chen2024near} in that we do not require a globally upper bounded variance, but only \Cref{asp: 2}.
\begin{lemma}\label{lem: iter_dif-m}
For problem \eqref{eq: mi}, let Assumption \ref{asp: 1} and \Cref{asp: 2} hold. We apply \Cref{alg:rec_reg} to solve \eqref{eq: delta_prob} with budget $T\geq 14175$ for number of oracles, $\mu=\frac{32\Lhat_F}{\sqrt{T}}\log_2^{3/2}\frac{\sqrt{T}}{48}$ and initial point $\zinit$. We have for all $s=1,\dots, S$ that
\begin{align*}
\mu_s^2 \mathbb{E} \| \bz_{s} - \bz_{s-1}^\star\|^2 \leq \left( \frac{1}{2} \right)^{\frac{Ts}{24\kappa S}} \mu^2 \|\zinit-\bz^\star_0\|^2 + \frac{2048S^3\widehat{G}^2}{T},
\end{align*}
where $\widehat{G}^2=B^2\| \zinit-\bz_0^\star\|^2 + G^2$, $\kappa = \frac{\Lhat_F}{\mu} = \frac{L_F+B}{\mu}$ and $S=\left\lfloor \log_2 \kappa \right\rfloor$.
\end{lemma}
\begin{proof}
At iteration $s\geq1$, we will invoke \Cref{lem: level2} with 
\begin{equation}
\by_0\leftarrow \bz_{s-1},~~~ \bx^\star_A\leftarrow \bz_{s-1}^{\star},~~~ A\leftarrow F^{s-1},~~~ T\leftarrow T/S,
\end{equation}
With these settings, we can verify that the requirements of \Cref{asp: inner} hold. Then, we will get the structural constants in the setting of \Cref{lem: level2} becoming
\begin{equation}
\mu_A\leftarrow (2^{s}-1)\mu_, ~~~ L_A \leftarrow L_{s-1} = L_F+(2^s-1)\mu, ~~~B_A \leftarrow B, ~~~ G_A \leftarrow G,
\end{equation}
where we use \Cref{fact:struc}.
This also tells us that $\frac{\Lhat_A}{\mu_A} = \frac{L_A+B_A}{\mu_A} \leq \frac{L_F+(2^s-1)\mu+B}{(2^s-1)\mu} \leq 1+\frac{L_F+B}{\mu} \leq \frac{3\kappa}{2}$ where $\kappa = \frac{L_F+B}{\mu}\geq2$ by Lemma \ref{lem:numer}(iii).

As a result, from \Cref{lem: level2}, we have for $s \geq 1$:
\begin{align}\label{eq: seq1_}
\mathbb{E} \| \bz_s - \bz_{s-1}^\star \|^2 &\leq \left(\frac{1}{2} \right)^{\frac{T}{12S\kappa}} \mathbb{E} \| \bz_{s-1} - \bz_{s-1}^\star \|^2 + \frac{192S\E(B^2\|\zinit-\bz^\star_{s-1}\|^2 +G^2)}{T\mu_{s-1}^2},
\end{align}
since $\mu_{s-1} \leq \mu_A$.
Here, we also majorized the constant in the last term for simplicity.

We next need to show that the term involving $\E\|\zinit-\bz_{s-1}^\star\|^2$ is not too large to deteriorate the bound. In particular, we use triangle inequality and \Cref{eq: nonexp1} to get
\begin{align*}
\| \zinit-\bz_{s-1}^\star\| &\leq \| \zinit - \bz_0^\star\| + \sum_{j=1}^{s-1} \| \bz_j^\star - \bz_{j-1}^\star\|\leq \| \zinit - \bz_0^\star\| +\sum_{j=1}^{s-1} \| \bz_j - \bz_{j-1}^\star\|.
\end{align*}
After taking the square of this estimate, we consequently have
\begin{align}\label{eq: sfj4}
\| \zinit-\bz_{s-1}^\star\|^2 \leq S\| \zinit - \bz_0^\star\|^2 + S\sum_{j=1}^{s-1} \| \bz_j-\bz_{j-1}^\star\|^2,
\end{align}
since $s\leq S$.

Using this bound and \Cref{eq: nonexp1} on \eqref{eq: seq1_} yields for $s\geq 2$,
\begin{align}\label{eq: seq1.5_}
\mathbb{E} \| \bz_s - \bz_{s-1}^\star \|^2 &\leq \left(\frac{1}{2} \right)^{\frac{T}{12S\kappa}} \mathbb{E} \| \bz_{s-1} - \bz_{s-2}^\star \|^2 + \frac{192S^2B^2(\|\zinit-\bz_0^\star\|^2 + \sum_{j=1}^{s-1}\E\|\bz_j-\bz_{j-1}^\star\|^2)}{T\mu_{s-1}^2} + \frac{192SG^2}{T\mu_{s-1}^2}.
\end{align}
Due to \Cref{lem:numer}\emph{(iii)}, the choice of $T$ and $\mu$ implies that 
\begin{equation*}
T \geq 48\kappa S, \qquad \mu^2 \geq \frac{1024S^3\Lhat_F^2}{T}.
\end{equation*}
As a result, we notice for the coefficients in \eqref{eq: seq1.5_} that
\begin{equation}\label{eq: numer_for_induction}
\left(\frac{1}{2} \right)^{\frac{T}{12S\kappa}} \leq \frac14\text{,~~~} \frac{192S^2B^2}{T\mu_{s-1}^2}\leq \frac{256S^2B^2}{T\mu^2} \leq \frac{1}{4S}, \text{~~~} \frac{192SG^2}{T\mu_{s-1}^2} \leq \frac{G^2}{4(L_F+B)^2S^2} \leq \frac{G^2}{4(L_F+B)^2},
\end{equation}
where we also used $\mu_{s-1} \geq \mu$ and $S\geq 1$.

We will next use induction. With these estimates, the inequality \eqref{eq: seq1.5_} becomes for $s\geq 2$
\begin{align}\label{eq: seq4}
\mathbb{E} \| \bz_s - \bz_{s-1}^\star \|^2 &\leq \frac{1}{4} \mathbb{E} \| \bz_{s-1} - \bz_{s-2}^\star \|^2 + \frac{\|\zinit - \bz_0^\star\|^2 + \sum_{j=1}^{s-1}\E \| \bz_j-\bz_{j-1}^\star\|^2}{4S} + \frac{G^2}{4(L_F+B)^2}.
\end{align}
For $s=1$, from \eqref{eq: seq1_}, with $\bz_0=\zinit$ and $\bz_{s-1}^\star = \bz_0^\star$, we have
\begin{align*}
\E \| \bz_1-\bz_0^\star\|^2 \leq \left( \frac{1}{2}\right)^{\frac{T}{12S\kappa}} \| \zinit - \bz_0^\star\|^2 + \frac{192S(B^2 \| \zinit-\bz_0^\star\|^2 + G^2)}{T\mu_0^2}.
\end{align*}
Applying \eqref{eq: sfj4} and \eqref{eq: numer_for_induction} gives
\begin{align*}
\E\|\bz_1-\bz_0^\star\|^2 \leq \frac{1}{4}\|\zinit-\bz_0^\star\|^2 + \frac{\|\zinit-\bz_0^\star\|^2}{4S} + \frac{G^2}{4(L_F+B)^2}.
\end{align*}
We will prove $\E\|\bz_j - \bz_{j-1}^\star\|^2 \leq \| \zinit-\bz_0^\star\|^2 + \frac{G^2}{(L_F+B)^2}$ for $j\geq 1$. By the last display equation, the claim holds at $j=1$. Assume that it holds for $j\leq s-1$ with $s\geq 2$. Then, \eqref{eq: seq4} gives
\begin{align}\label{eq: seq5}
\mathbb{E} \| \bz_s - \bz_{s-1}^\star \|^2 
&\leq \frac{1}{2} \left(\| \zinit-\bz_0^\star\|^2 + \frac{G^2}{(L_F+B)^2}\right) + \frac{G^2}{4(L_F+B)^2} \notag \\
&\leq \| \zinit-\bz_0^\star\|^2 + \frac{G^2}{(L_F+B)^2}.
\end{align}
This completes the induction to derive the global upper bound on $\mathbb{E} \| \bz_s-\bz_{s-1}^\star\|^2$.

Then, using \eqref{eq: sfj4} and taking expectation, we have
\begin{align*}
\E\| \zinit-\bz_{s-1}^\star\|^2 &\leq S\| \zinit - \bz_0^\star\|^2 + S\sum_{j=1}^{s-1}\left( \| \zinit-\bz_0^\star\|^2 + \frac{G^2}{(L_F+B)^2}\right)\\
&\leq 2S^2 \| \zinit-\bz_0^\star\|^2 + \frac{S^2 G^2}{(L_F+B)^2}.
\end{align*}
We now plug this bound to \eqref{eq: seq1_} which gives us the bound
\begin{align}\label{eq: seq2_}
\mathbb{E} \| \bz_s - \bz_{s-1}^\star \|^2 &\leq \left(\frac{1}{2} \right)^{\frac{T}{12S\kappa}} \mathbb{E} \| \bz_{s-1} - \bz_{s-1}^\star \|^2 + \frac{192S(B^2(2S^2 \| \zinit-\bz_0^\star\|^2 + \frac{S^2 G^2}{(L_F+B)^2}) +G^2)}{T\mu_{s-1}^2} \\
&\leq \left(\frac{1}{2} \right)^{\frac{T}{12S\kappa}} \mathbb{E} \| \bz_{s-1} - \bz_{s-1}^\star \|^2 + \frac{384S^3(B^2\| \zinit-\bz_0^\star\|^2 + G^2)}{T\mu_{s-1}^2}
\end{align}
We follow the exact same steps that followed \eqref{eq: seq1} (where the difference is having $384S^3$ instead of $75S$ and $B^2\|\zinit-\bz_0^\star\|^2+G^2$ instead of $G^2$ on the last term) to get the assertion.
\end{proof}
Let us remark that the induction argument in this proof to bound $\E\|\zinit-\bz_{s-1}^\star\|^2$ can be strengthened with a more involved analysis to improve the order of the $S$ term in the final bound. However we chose the current argument for simplicity, since this order only affects the order of the logarithmic terms in the final bound in the sequel.

The next result corresponds to \Cref{th: hdelta} and \Cref{th: cc} without the bounded variance assumption.
\begin{theorem}\label{th: unbdd_monot}
For problem \eqref{eq: mi}, let Assumptions \ref{asp: 1} and \ref{asp: 2} hold. We apply \Cref{alg:rec_reg} to solve \eqref{eq: delta_prob} with 
\begin{equation*}
\mu = \frac{32\Lhat_F}{\sqrt{T}}\log_2^{3/2}\left( \frac{\sqrt{T}}{48} \right),~~~\eta=\frac{1}{3\Lhat_F},~~~S=\lfloor \log_2(\Lhat_F/\mu) \rfloor,~~~T \geq 14175,
\end{equation*}
and initial point $\zinit$. Then, we have that 
\begin{equation*}
\mathbb{E} \| \cG_{\eta, H}(\bz_S) \| = O\left( \mu D_\star + \left(\frac{1}{2}\right)^{\frac{T}{48S\Lhat_F/\mu}} \mu D_\star + \frac{S^{5/2}\left(BD_\star+G\right)}{\sqrt{T}} \right).
\end{equation*}
\end{theorem}
\begin{remark}
By the value of $\mu=\widetilde{O}\left(\frac{1}{\sqrt{T}} \right)$ and $T \geq 48S\Lhat_F/\mu$ as shown in Lemma \ref{lem:numer}, one can estimate that the complexity will be of the order $\varepsilon^{-2}$ up to logarithmic terms. We follow by a corollary for the order of the terms and then we give the precise constants.
\end{remark}
\begin{proof}
The result of \Cref{lem: iter_dif-m} shows that the hypothesis of \Cref{th: it_dif_to_gradmapmu} is satisfied with
\begin{equation}\label{eq: ghat_def}
C=2048S^2 \text{~~~and~~~} \widehat{G}^2=B^2 \| \zinit-\bz_0^\star\|^2  +G^2.
\end{equation}
We also know that applying \Cref{lem: reg_to_unreg} yields the inequality
\begin{align}\label{eq: 512}
\| \cG_{\eta, H}(\bz_S)\| \leq 2.5 \sum_{j=1}^{S} \mu_j \| \bz_j - \bz_{j-1}^\star\| + 9\Lhat_F \| \bz_S - \bz_{S-1}^\star\| + \mu \|\zinit-\bx^\star\|.
\end{align}
Using \eqref{eq: sll5} and \eqref{eq: bdd_lastterm} and \Cref{eq: nonexp1} for $\|\zinit - \bz^\star_0\|\leq \| \zinit - \bx^\star\|=D_\star$ to estimate the right-hand side of \eqref{eq: 512} gives the claimed bound.
\end{proof}
\begin{corollary}
Under the same setup of \Cref{th: unbdd_monot}, with $\mu = \frac{32\Lhat_F}{\sqrt{T}}\log_2^{3/2}\left( \frac{\sqrt{T}}{48} \right)$, $T\geq 14175$, we have that 
\begin{equation}\label{eq: unbdd_ratebd}
\mathbb{E}\|\cG_{\eta, H}(\bz_S)\|= O\left( \frac{\Lhat_FD_\star (\ln \sqrt{T})^{3/2}}{\sqrt{T}} + \frac{(\ln\sqrt{T})^{5/2}(BD_\star+G)}{\sqrt{T}} \right).
\end{equation}
Consequently, to obtain
\begin{equation}
\mathbb{E}\|\cG_{\eta, H}(\bz_S)\|\leq\varepsilon,
\end{equation}
the number of required stochastic first-order oracles $T$ is upper bounded by
\begin{equation*}
O\left( 1+\frac{(L_F^2+B^2)D_\star^2 + G^2}{\varepsilon^2}\ln^5\left(\frac{D_\star(L_F+B)+G}{\varepsilon} +1\right)\right),
\end{equation*}
where we only suppress the absolute constants.
\end{corollary}
\begin{proof}
We next write down the precise bound for the gradient mapping norm, unpacking the bound in \Cref{th: unbdd_monot}.
\begin{align}\label{eq: vmk4}
\mathbb{E} \| \cG_{\eta, H}(\bz_S) \| \leq \mu D_\star + 5\left(\frac{1}{2}\right)^{\frac{T}{48\kappa S}}\mu D_\star + \frac{114S^{5/2}\Ghat}{\sqrt{T}} + 18\left( \left(\frac{1}{2}\right)^{\frac{T}{48\kappa}} \mu D_\star + \frac{46S^{3/2}\Ghat}{\sqrt{T}} \right),
\end{align}
where $\Ghat$ is as defined in \eqref{eq: ghat_def}, that is, $\widehat{G}^2=B^2 \| \zinit-\bz_0^\star\|^2  +G^2$. Plugging in the definition of $\mu$ gives \eqref{eq: unbdd_ratebd}.

We first estimate the first term on the right-hand side. In particular, we have
\begin{align*}
\mu D_\star = \frac{32\Lhat_FD_\star}{\sqrt{T}}\left( \log_2\frac{\sqrt{T}}{48} \right)^{3/2}.
\end{align*}
Then, by using \Cref{lem:numer}(i) with $a=32\Lhat_F D_\star$, $b=\frac{1}{48}$, and $c=\frac{3}{2}$, we have that $\mu D_\star \leq\varepsilon$ whenever 
\begin{equation}\label{eq: kjh4}
T = \Theta\left(\frac{\Lhat_F^2 D_\star^2}{\varepsilon^2}\ln^3\left( \frac{\Lhat_FD_\star}{\varepsilon}+1 \right)\right).
\end{equation}
As long as we have $\mu D_\star \leq \varepsilon$, we also have for the second-term on the right-hand side of \eqref{eq: vmk4}
\begin{align*}
5\left(\frac{1}{2}\right)^{\frac{T}{48\kappa S}}\mu\| \zinit-\bz_0^\star\| \leq 2.5\varepsilon,
\end{align*}
because $T\geq 48\kappa S$, see Lemma \ref{lem:numer}(iii).

Similarly, the fourth term on the right-hand side of \eqref{eq: vmk4} is upper bounded by these terms up to absolute constants, as a result, it will be also of the order $\varepsilon$ scaled by an absolute constant.

As shown in Lemma \ref{lem:numer}(iii), we have
\begin{align*}
S^{5/2} \leq \left(\log_2\frac{\sqrt{T}}{48}\right)^{5/2}.
\end{align*}
Then, we estimate the third term on the right-hand side of \eqref{eq: vmk4} as
\begin{align}\label{eq: kjh5}
\frac{114S^{5/2}\Ghat}{\sqrt{T}} \leq \varepsilon, \text{~when~} T = \Theta\left( \frac{B^2D_\star^2+G^2}{\varepsilon^2}\ln^5\left(\frac{BD_\star+G}{\varepsilon}+1\right) \right),
\end{align}
where the calculation of $T$ uses \Cref{lem:numer}(i) with $a=114\Ghat$, $b=\frac{1}{48}$, and $c=5/2$. We also used $\Ghat^2 \leq B^2D_\star^2+G^2$ and $\Ghat \leq BD_\star+G$, due to Lemma \ref{eq: nonexp1}. The last term on the right-hand side of \eqref{eq: vmk4} is dominated up to constants by the third term so the given $T$ is sufficient (up to constants) to make the last term on the right-hand side of \eqref{eq: vmk4} less than $\varepsilon$.

Let us recall that $T\geq 48\kappa S$ was proved in Lemma \ref{lem:numer}(iii). Combining the estimates for $T$ in \eqref{eq: kjh4} and \eqref{eq: kjh5} gives the assertion.
\end{proof}

\section{Auxiliary Lemmas}
In this section, we will prove auxiliary results that were used in the main proofs. The first result connects the solutions of subproblems for different regularization levels.

Given $F^0=F^\mu$, let us recall that at iteration $s\geq 1$, we estimate $\bz^\star_{s-1}$ such that $0\in (F^{s-1}+\partial r)(\bz^\star_{s-1})$ where
\begin{align}\label{eq: reg_defs}
F^{s-1}(\bz)= F^\mu(\bz)+ \sum_{i=1}^{s-1} \mu_i(\bz-\bz_i), \text{~for~} s=1,2,\dots, S
\end{align}
with $\bz^\star_0$ is as defined in \eqref{eq: delta_prob}.
\begin{lemma}\label{eq: nonexp1}
For $s=1,\dots, S$, if we are given that $F^0=F^\mu$ (see \eqref{eq: fdelta}), $\bz_s^\star = (F^{s}+\partial r)^{-1}(0)$ and $\bz_{s-1}^\star = (F^{s-1}+\partial r)^{-1}(0)$, where
\begin{equation}\label{eq: sce4}
F^s(\bz) = F^{s-1}(\bz) + \mu_s(\bz-\bz_s),
\end{equation}
it follows that
\begin{equation*}
\| \bz_s-\bz_{s}^\star\| \leq \| \bz_s-\bz_{s-1}^\star\| \text{~and~} \| \bz_s^\star - \bz_{s-1}^\star \| \leq \| \bz_s - \bz_{s-1}^\star\|.
\end{equation*}
Moreover, we have that $\|\zinit - \bz^\star_0\|\leq \| \zinit - \bx^\star\|$.
\end{lemma}
\begin{proof}
By definition of the solutions, we have that
\begin{equation*}
0\in(F^s+\partial r)(\bz^\star_s) \text{~and~} 0\in (F^{s-1}+\partial r)(\bz_{s-1}^\star) \text{~where~} F^s(\bz_{s-1}^\star)  = F^{s-1}(\bz_{s-1}^\star) + \mu_s (\bz_{s-1}^\star - \bz_{s}),
\end{equation*}
because of \eqref{eq: sce4}. Consequently, we have that $\mu_s (\bz_{s-1}^\star - \bz_s)\in(F^s+\partial r)(\bz_{s-1}^\star)$.

Let us now use that $(F^s+\partial r)$ is $\mu_s$-strongly monotone to write
\begin{align*}
&\mu_s\langle \bz_s - \bz_{s-1}^\star, \bz_s^\star-  \bz_{s-1}^\star \rangle \geq \mu_s \| \bz_s^\star-\bz_{s-1}^\star\|^2 \\
\iff& \frac{1}{2} \left( \|\bz_s-\bz_{s-1}^\star\|^2 + \|\bz_s^\star - \bz_{s-1}^\star \|^2 - \| \bz_s-\bz_s^\star\|^2 \right) \geq \| \bz_s^\star - \bz_{s-1}^\star\|^2.
\end{align*}
Rearranging the final inequality gives both results.

In view of \eqref{eq: fdelta}, if $F$ is $\mu$-strongly monotone, the last claim holds with equality. Otherwise, we use $\mu$-strong monotonicity of $F^\mu+\partial r$ as in the first part of the proof, along with
\begin{align*}
0\in (F^\mu + \partial r)(\bz_0^\star), ~~~0\in (F + \partial r)(\bx^\star), ~~~ F^\mu(\bx^\star)= F(\bx^\star) + \mu(\bx^\star-\zinit).
\end{align*}
This establishes the last inequality in the statement.
\end{proof}
This lemma generalizes \citep[Lemma A.1]{chen2024near} which was specialized to unconstrained min-max problems. The extension is to handle the multi-valued mapping $\partial r$ in the definition of subproblems. For this, we keep all the terms on the right-hand side to depend on the quality of approximate subproblem solutions.

\begin{lemma}\label{eq: grad_map_sym_bd}
Under Assumptions \ref{asp: 1}, \ref{asp: 2} and the setup of \eqref{eq: delta_prob}--\eqref{eq: fsdef}, set $\eta =\frac{1}{3\widehat{L}_F}$ where $\Lhat_F = L_F+B$. Then, we have
\begin{equation}\label{eq: kop4}
\|\cG_{\eta, H^\mu}(\bz_S) \| \leq 2 \sum_{j=1}^{S} \mu_j \| \bz_j - \bz_{j-1}^\star\| + 9\Lhat_F \| \bz_S - \bz_{S-1}^\star\|,
\end{equation}
where $\mu_s = 2^s\mu$ and $\bz_s, \bz_{s}^\star$ are as given in \eqref{eq: reg_defs} (see also \eqref{eq: fsdef} and \eqref{eq: delta_prob}).
\end{lemma}
\begin{proof}
We first recall the definition of the residual for operators $H^\mu=F^\mu+\partial r$ and $H^s = F^s+\partial r$, respectively:
\begin{align*}
\cG_{\eta, H^{\mu}}(\bz) = \eta^{-1}\left( \bz-\prox_{\eta r}(\bz-\eta F^\mu(\bz)) \right) \text{~~~and~~~}
\cG_{\eta, H^{S-1}}(\bz) = \eta^{-1}\left( \bz-\prox_{\eta r}(\bz-\eta F^{S-1}(\bz)) \right).
\end{align*}
By these definitions, we have
\begin{align}
\|\cG_{\eta, H^\mu}(\bz_S) - \cG_{\eta, H^{S-1}}(\bz_S) \| &=  \eta^{-1}\| \prox_{\eta r}(\bz_S-\eta F^\mu(\bz_S)) - \prox_{\eta r}(\bz_S - \eta F^{S-1}(\bz_S)) \| \notag \\
&\leq  \| F^\mu(\bz_S) - F^{S-1}(\bz_S)\| \notag \\
&\leq  \sum_{i=1}^{S-1}\mu_i \|\bz_S - \bz_i\|,\label{eq: xso3}
\end{align}
where the first inequality used nonexpansiveness of the proximal operator, see for example \citep[Theorem 6.42]{beck2017first}. The last inequality is by \eqref{eq: reg_defs} and triangle inequality.

Now we estimate the term on the right-hand side by using triangle inequality and Lemma \ref{eq: nonexp1}
\begin{align*}
\| \bz_{S}-\bz_i\| &\leq \| \bz_S - \bz_{S-1}^\star\| + \sum_{j=i}^{S-2} \|\bz_{j+1}^\star - \bz_j^\star\| + \| \bz_i^\star - \bz_i\|\leq \sum_{j=i}^S \| \bz_j - \bz_{j-1}^\star\|.
\end{align*}
Substituting this and changing the order of summations give
\begin{align*}
\sum_{i=1}^{S-1} \mu_i \| \bz_S - \bz_i\| &\leq \sum_{i=1}^{S-1} \mu_i\sum_{j=i}^S \| \bz_j - \bz_{j-1}^\star\| \\
&= \sum_{j=1}^{S-1} \sum_{i=1}^j \mu_i \| \bz_j-\bz_{j-1}^\star\| + \sum_{i=1}^{S-1} \mu_i \| \bz_S-\bz_{S-1}^\star\|.
\end{align*}
By using the setting $\mu_i=2^i \mu$, we have $\sum_{i=1}^j \mu_i= \sum_{i=1}^j 2^i \mu \leq 2^{j+1}\mu = \mu_{j+1} = 2\mu_j$ and $\sum_{i=1}^{S-1}\mu_i \leq \mu_S$, we get
\begin{align*}
\sum_{i=1}^{S-1} \mu_i \| \bz_S - \bz_i\| &\leq 2\sum_{j=1}^S \mu_j \| \bz_j-\bz_{j-1}^\star\|.
\end{align*}
Using this in \eqref{eq: xso3} gives
\begin{align}
 \|\cG_{\eta, H^\mu}(\bz_S) \| \leq  \|\cG_{\eta, H^{S-1}}(\bz_S)\| +  2\sum_{j=0}^{S-1}\mu_{j+1} \| \bz_{j+1} - \bz_j^\star\|,\label{eq: grad_map_end}
\end{align}
where we used the triangle inequality on the left-hand side of \eqref{eq: xso3}.

Next, we estimate $\|\cG_{\eta, H^{S-1}}(\bz_S)\|$. First, by the definition of $\bz_{S-1}^{\star}$ (see \eqref{eq: reg_defs}), we have
\begin{align*}
\cG_{\eta, H^{S-1}}(\bz_{S-1}^\star) = 0 &\iff \bz_{S-1}^\star = \prox_{\eta r}(\bz_{S-1}^\star - \eta F^{S-1}(\bz_{S-1}^\star)) \\
&\iff \bz_{S-1}^\star + \eta \partial r(\bz_{S-1}^\star) \ni \bz_{S-1}^\star - \eta F^{S-1}(\bz_{S-1}^\star) \\
&\iff 0\in (\partial r+F^{S-1})(\bz_{S-1}^\star).
\end{align*}
Using this, we estimate as
\begin{align*}
\|\cG_{\eta, H^{S-1}}(\bz_S)\| &= \|\cG_{\eta, H^{S-1}}(\bz_S) - \cG_{\eta, H^{S-1}}(\bz_{S-1}^\star)\| \\
&\leq \eta^{-1} \left( \| \bz_S - \bz_{S-1}^\star\| +\|\prox_{\eta r}(\bz_S - \eta F^{S-1}(\bz_S))-\prox_{\eta r}(\bz_{S-1}^\star - \eta F^{S-1}(\bz_{S-1}^\star))\|  \right) \\
&\leq \eta^{-1} \left( 2\| \bz_S - \bz_{S-1}^\star\| +\eta \| F^{S-1}(\bz_{S})- F^{S-1}(\bz_{S-1}^\star)\|  \right) \\
&\leq \eta^{-1} \left( 2\| \bz_S - \bz_{S-1}^\star\| +\eta L_{S-1} \| \bz_{S} -\bz_{S-1}^\star\|  \right),
\end{align*}
where the first step is by triangle inequality, the second by triangle inequality and nonexpansiveness of the proximal operator, the third by the $L_{S-1}$-Lipschitzness of $F^{S-1}$. Since $L_{S-1}\leq 2\Lhat_F$, we insert this into \eqref{eq: grad_map_end} to obtain the bound.
\end{proof}
This lemma is extending \citep[Lemma 3.1]{chen2024near} to using gradient mapping instead of the gradient and the existence of the regularizer $r$ in the subproblems. In contrast to \citep[Lemma 5.1]{allen2018make}, we cannot rely on inequalities specific to minimization such as the descent lemma.
\begin{lemma}\label{lem: reg_to_unreg}
Let $F$ be monotone and $F^\mu, (\bz_s^\star), (\bz_s)$ defined as \eqref{eq: fdelta}, \eqref{eq: fsdef}, and Algorithm \ref{alg:rec_reg}, respectively. Set $\eta =\frac{1}{3\widehat{L}_F}$ where $\Lhat_F = L_F+B$. Then, we have
\begin{align}
\| \cG_{\eta, H}(\bz_S)\|  \leq 2.5 \sum_{j=1}^{S} \mu_j \| \bz_j - \bz_{j-1}^\star\| + 9\Lhat_F \| \bz_S - \bz_{S-1}^\star\| + \mu \|\zinit-\bx^\star\|.\label{eq: nbv4}
\end{align}
\end{lemma}
\begin{proof}
Since $F$ is only monotone, we are in the second case of \eqref{eq: fdelta}. First, we characterize the difference between the residual for solving the original problem \eqref{eq: mi} and the perturbed problem \eqref{eq: delta_prob}.

By the definition of the residual, we have
\begin{align*}
\cG_{\eta, H}(\bz) = \eta^{-1}\left( \bz-\prox_{\eta r}(\bz-\eta F(\bz)) \right) \text{~~~and~~~} \cG_{\eta, H^{\mu}}(\bz) = \eta^{-1}\left( \bz-\prox_{\eta r}(\bz-\eta F^{\mu}(\bz)) \right).
\end{align*}
With this, nonexpansiveness of the proximal operator and the triangle inequality, we obtain
\begin{align}
\|\cG_{\eta, H}(\bz_S) - \cG_{\eta, H^\mu}(\bz_S)\| \leq \|F(\bz_S) - F^{\mu}(\bz_S)\| \leq \mu\|\bz_S -\zinit\|\leq \mu(\|\bz_S -\bz_0^\star\| + \| \bz_0^\star - \zinit\|),\label{eq: xid4}
\end{align}
where $\bz_0^\star$ is the solution of the perturbed problem as in \eqref{eq: delta_prob} and the second inequality uses the definition in the second case of \eqref{eq: fdelta}. For the second term on the right-hand side of \eqref{eq: xid4}, we can use \Cref{eq: nonexp1} to get
\begin{equation}\label{eq:init_dist_sol}
\| \zinit - \bz^\star_0 \| \leq \|\zinit - \bx^\star\|.
\end{equation}
For the first term on the right-hand side of \eqref{eq: xid4}, we estimate similar to the proof of the \Cref{eq: grad_map_sym_bd}:
\begin{align*}
\| \bz_S - \bz_0^\star\| &\leq \| \bz_S - \bz_{S-1}^\star \| + \sum_{j=1}^{S-1} \| \bz_{j}^\star - \bz_{j-1}^\star\|
\leq \| \bz_S - \bz_{S-1}^\star \| + \sum_{j=1}^{S-1} \| \bz_{j} - \bz_{j-1}^\star\| 
= \sum_{j=1}^{S} \| \bz_{j} - \bz_{j-1}^\star\|,
\end{align*}
where the second step used \Cref{eq: nonexp1}.

On \eqref{eq: xid4}, we substitute the last two estimates to the right-hand side and use triangle inequality on the left-hand side to deduce
\begin{align}
\| \cG_{\eta, H}(\bz_S)\| &\leq \| \cG_{\eta, H^\mu}(\bz_S)\| + \mu \sum_{j=1}^S\|\bz_j-\bz_{j-1}^\star\| + \mu \|\zinit-\bx^\star\|\notag \\
&\leq 2.5 \sum_{j=1}^{S} \mu_j \| \bz_j - \bz_{j-1}^\star\| + 9\Lhat_F \| \bz_S - \bz_{S-1}^\star\| + \mu \|\zinit-\bx^\star\|,\label{eq: nbv4}
\end{align}
where the last step is because of $\mu_j=2^j\mu\Rightarrow \mu\leq \mu_j/2$ and \Cref{eq: grad_map_sym_bd}.
\end{proof}

\begin{lemma}\label{th: it_dif_to_gradmapmu}
For problem \eqref{eq: mi}, let Assumptions \ref{asp: 1} and \ref{asp: 2} hold. We apply \Cref{alg:rec_reg} to solve \eqref{eq: delta_prob} with budget $T \geq 48S\kappa=48S\Lhat_F/\mu$ for the number of oracles where $S = \left\lfloor \log_2\frac{\Lhat_F}{\mu} \right\rfloor \geq 1$ and initial point $\zinit$. Set $\eta =\frac{1}{3\widehat{L}_F}$ where $\Lhat_F = L_F+B$. If we have an estimate of the form
\begin{align*}
\mu_s^2 \mathbb{E} \| \bz_{s} - \bz_{s-1}^\star\|^2 \leq \left( \frac{1}{2} \right)^{\frac{Ts}{24\kappa S}} \mu^2 \|\zinit-\bz^\star_0\|^2 + \frac{CS\Ghat^2}{T},
\end{align*}
for some $C, \Ghat$, then it follows that
\begin{align*}
\mathbb{E} \| \cG_{\eta, H^\mu}(\bz_S)\| \leq \left( 4\left(\frac{1}{2}\right)^{\frac{T}{48\kappa S}}+18\left(\frac{1}{2}\right)^{\frac{T}{48\kappa}} \right)\mu\|\zinit - \bz_0^\star\| + \frac{\sqrt{C}\Ghat}{\sqrt{T}}\left( 2S^{3/2}+18\sqrt{S} \right).
\end{align*}
\end{lemma}
\begin{proof}
Let us recall that the hypothesis of the lemma implies for all $s=1,\dots, S$ that
\begin{equation}\label{eq: sll4}
\mu_s \mathbb{E} \| \bz_{s} - \bz_{s-1}^\star\| \leq \left( \frac{1}{2} \right)^{\frac{Ts}{48\kappa S}} \mu \|\zinit-\bz^\star_0\| + \frac{\sqrt{CS}\Ghat}{\sqrt{T}},
\end{equation}
by Jensen's inequality and concavity of the square root. Summing up both sides of the inequality gives
\begin{align}\label{eq: sll5}
\sum_{s=1}^S \mu_s \mathbb{E} \| \bz_{s} - \bz_{s-1}^\star\| \leq   \mu \|\zinit-\bz^\star_0\|\sum_{s=1}^S \left( \frac{1}{2} \right)^{\frac{Ts}{48\kappa S}}+ \frac{\sqrt{C}S^{3/2}\Ghat}{\sqrt{T}} \leq 2\left(\frac{1}{2}\right)^{\frac{T}{48\kappa S}}\mu\| \zinit-\bz_0^\star\| + \frac{\sqrt{C}S^{3/2}\Ghat}{\sqrt{T}},
\end{align}
where the last estimate is summing the geometric series: since $T\geq 48\kappa S$, we have $\sum_{j=1}^s\left(\frac{1}{2} \right)^{Tj/(48\kappa S)}\leq 2\left(\frac{1}{2}\right)^{\frac{T}{48\kappa S}}$. This is because $\sum_{j=1}^s\alpha^j = \frac{\alpha(1-\alpha^S)}{1-\alpha} \leq \frac{\alpha}{1-\alpha}$ where $\alpha =\left(\frac{1}{2} \right)^{T/(48\kappa S)}\leq1/2$.

In view of the second term on the right-hand side of the result of \Cref{eq: grad_map_sym_bd}, we have
\begin{equation}\label{eq: bdd_lastterm}
9\Lhat_F\E\|\bz_S-\bz_{S-1}^\star\| \leq 18\mu_S \E\|\bz_S-\bz_{S-1}^\star\| \leq 18\left( \left(\frac{1}{2}\right)^{\frac{T}{48\kappa}} \mu\| \zinit - \bz_0^\star\| + \frac{\sqrt{SC}\Ghat}{\sqrt{T}} \right),
\end{equation}
where the first inequality is because of $2^{S+1} \geq \frac{\Lhat_F}{\mu}$, which is by the definition of $S$ and $\mu_S = 2^S \mu$. The second inequality is substituting \eqref{eq: sll4} with $s=S$.

Using \eqref{eq: sll5} and \eqref{eq: bdd_lastterm} to estimate the right-hand side of \Cref{eq: grad_map_sym_bd} gives the claimed bound.
\end{proof}

\begin{fact}\label{fact:struc}
Let $F$ satisfy \Cref{asp: 2} with constants $B, G$. Then, $F^s$ defined in \eqref{eq: fsdef}, satisfies the BG assumption with the same constants. Moreover, $F^{s-1}$ is $(L_F+(2^s-1) \mu)$-Lipschitz, $((2^s-1)\mu)$-strongly monotone.
\end{fact}
\begin{proof}
Since $F^s(\bz) = F^\mu(\bz) + \sum_{i=1}^s \mu_i(\bz-\bz_i)$ and $F^s_\xi(\bz) = F^\mu_\xi(\bz) + \sum_{i=1}^s \mu_i(\bz-\bz_i)$, we have that
\begin{align*}
\mathbb{E}\| F^s_\xi(\bz) - F^s(\bz)\|^2 = \E\| F_\xi(\bz) - F(\bz)\|^2 \leq B^2\| \bz-\zinit\|^2 + G^2.
\end{align*}
This gives the first assertion.

For the second, by definition, we have $F^{s-1}(\bz)=F^\mu(\bz)+\sum_{i=1}^{s-1} \mu_i(\bz-\bz_i)$ which is $L_F+(2^s-1)\mu$-Lipschitz since $\sum_{i=1}^{s-1} \mu_i = \sum_{i=1}^{s-1} 2^i\mu = (2^s-2)\mu$. The strong monotonicity constant is calculated in the same way.
\end{proof}
For getting precise constants for the complexity results and the parameters (see for example \Cref{subsec:monot_bddvar}), we often solve inequalities of the form $e^x \geq Cx$. For simplicity, we use the following simple sufficient condition to satisfy these inequalities without resorting to the Lambert W function to find exact roots.
\begin{lemma}{\citep[Lemma A.2]{shalev2014understanding}}\label{lem: tr_lb}
For some $C> 1$, if $x\geq 4\ln2 + 2 \ln C$, then we have $x\geq \ln x+\ln C \iff e^x \geq Cx$.
\end{lemma}
We now continue to derive numerical sufficient conditions for our complexity results.
\begin{lemma}\label{lem:numer}
\begin{enumerate}[i.]
\item Let $a,b,c, \varepsilon$ be positive real numbers and $T\geq \frac{1}{b^2}$. Then, the inequality $\frac{a}{\sqrt{T}}\left(\log_2b\sqrt{T}\right)^c\leq \varepsilon$ holds if $T \geq \frac{a^2c^{2c}}{\varepsilon^2(\ln 2)^{2c}}\left( 4\ln2+2 \ln \left(\frac{c (ab)^{1/c}}{\varepsilon^{1/c}\ln 2}+1\right) \right)^{2c}$.
\item Let $a,b,c,\varepsilon'$ be positive real numbers and $T\geq \frac{1}{b}$. Then, the inequality $\frac{a}{T}\left(\log_2bT\right)^c\leq \varepsilon'$ holds if $T \geq \frac{ac^{c}}{\varepsilon'(\ln 2)^{c}}\left( 4\ln2+2 \ln\left( \frac{c (ab)^{1/c}}{(\varepsilon')^{1/c}\ln 2}+1\right) \right)^{c}$.
\item If $T\geq 14175$ and $\mu=\frac{32\Lhat_F}{\sqrt{T}}\log_2^{3/2}\frac{\sqrt{T}}{48}$, then the following inequalities hold: $T \geq 48\kappa S$, $\mu^2 \geq \frac{1024S^3\Lhat_F^2}{T}$, where $\kappa =\frac{\Lhat_F}{\mu}$ and $S=\lfloor \log_2\frac{\Lhat_F}{\mu} \rfloor$. Moreover, we have $S \leq  \log_2\frac{\sqrt{T}}{48}$, $\kappa \geq 2$, and $S\geq 1$.
\end{enumerate}
\end{lemma}
\begin{proof}
\begin{enumerate}
\item By the required conditions on $a, b, c, T$, we have that the required inequality is equivalent to
\begin{equation}
\log_2 b\sqrt{T} \leq \left(\frac{\varepsilon \sqrt{T}}{a}\right)^{1/c}.\label{eq: scid4}
\end{equation}
Let us define $x$ such that
\begin{equation}\label{eq: sov4}
\left(\frac{\varepsilon \sqrt{T}}{a}\right)^{1/c} = \frac{cx}{\ln 2} \iff T^{\frac{1}{2c}} = \frac{cxa^{1/c}}{\varepsilon^{1/c}\ln 2}.
\end{equation}
We also note that
\begin{equation*}
\log_2 b\sqrt{T} = \frac{c}{\ln 2} \ln (b^{\frac{1}{c}}T^{\frac{1}{2c}}).
\end{equation*}
Plugging in the last two estimates into \eqref{eq: scid4} gives 
\begin{equation*}
\ln (b^{1/c}T^{1/(2c)}) \leq x \iff b^{1/c} T^{1/(2c)} \leq e^x \iff \frac{c(ab)^{1/c}}{\varepsilon^{1/c}\ln2}x \leq e^x \Leftarrow \left(\frac{c(ab)^{1/c}}{\varepsilon^{1/c}\ln2}+1\right)x \leq e^x,
\end{equation*}
since $x \geq 0$.

From \Cref{lem: tr_lb}, we know that this is satisfied as long as
\begin{equation*}
x \geq 4\ln 2 + 2 \ln \left(\frac{c(ab)^{1/c}}{\varepsilon^{1/c}\ln 2}+1\right),
\end{equation*}
which is translated to a lower bound on $T$ given in the assertion by using the last identity in \eqref{eq: sov4}

\item We use the proof of the first assertion by replacing $T \leftarrow T^2$ and $\varepsilon\leftarrow \varepsilon'$.
\item 
We first show that $S\geq 1$ for which it is enough to show $\kappa \geq 2$. Let $x = \log_2\frac{\sqrt{T}}{48} \iff \sqrt{T}=48\times 2^x$ and hence $\kappa = \frac{\sqrt{T}}{32\log_2^{3/2}\frac{\sqrt{T}}{48}} = \frac{\sqrt{T}}{32x^{3/2}}$. The inequality $\kappa \geq 2$ is equivalent to
\begin{align*}
2^x \geq \frac{4}{3} x^{3/2} \iff \frac{(2^{2/3})^x}{x} \geq \left(\frac{4}{3}\right)^{2/3}.
\end{align*}
Denoting $\beta = 2^{2/3}>1$, we have by a standard calculation that $\min_{x> 0} \beta^x/x=e\ln \beta > 1.25 > 1.22 > \left(\frac{4}{3}\right)^{2/3}$.

We then prove the assertion $\mu^2 T \geq 1024S^3\Lhat_F^2$. By a straightforward calculation, we have
\begin{align*}
&\mu^2 T = 1024\Lhat_F^2\log_2^3\frac{\sqrt{T}}{48}  \geq 1024S^3\Lhat_F^2= 1024\Lhat_F^2\left\lfloor \log_2 \frac{\Lhat_F}{\mu}\right\rfloor^{3} \Leftarrow \log_2 \frac{\sqrt{T}}{48} \geq \log_2\frac{\sqrt{T}}{32\log_2^{3/2}\frac{\sqrt{T}}{48}} \\
&\iff \log_2^{3/2}\frac{\sqrt{T}}{48} \geq \frac{3}{2} \Leftarrow T \geq 14175.
\end{align*}
Recalling $x = \log_2\frac{\sqrt{T}}{48}$, we have just established that $x\geq S$. Then, since $\kappa = \frac{\sqrt{T}}{32\log_2^{3/2}\frac{\sqrt{T}}{48}} = \frac{\sqrt{T}}{32x^{3/2}}$, we have
\begin{align*}
48\kappa S \leq 48\kappa x \leq \frac{3\sqrt{T}}{2\sqrt{x}} \leq \frac{3\sqrt{T}}{2} \leq T,
\end{align*}
since $x \geq 1$ and $\sqrt{T} \geq 3/2$.
\end{enumerate}
\end{proof}
\section{Conclusions}
Our goal in this work was to obtain a complexity result for the gradient mapping that is optimal up to logarithmic terms for solving constrained convex-concave min-max problems or stochastic VIs. 

Our construction relied on the idea of recursive regularization of \cite{allen2018make} to get the best complexity, resulting in a three-loop construction. A clear direction for future research is to obtain the same complexity with a simpler approach, ideally with a single-loop algorithm. We believe that one may be able to go for a two-loop construction by combining \cite{kotsalis2022simple} and \cite{allen2018make}, although it is not clear to us if this construction would still be able to handle the BG assumption. Obtaining the same complexity with a single-loop algorithm is open to our knowledge, even for minimization problems.

The second direction is analyzing the monotone case directly, rather than using perturbation by a small parameter $\mu$. This is necessary even for minimization problems to obtain the best gradient-norm guarantee to our knowledge, so we believe it is a question with a wider scope than min-max problems. Third, we believe that extension of our results to general monotone inclusions is possible by changing the inner solver. Finally, by using similar ideas to \cite{chen2024near}, we suspect it would be possible to allow some form of nonmonotonicity, to solve problems with cohypomonotone operators and constraints.

\bibliographystyle{alpha}
\bibliography{opt_stocminmax}
\end{document}